\UseRawInputEncoding
\documentclass{amsart}
\usepackage{url}
\usepackage{amssymb}
\usepackage{verbatim}
\usepackage{hyperref}
\allowdisplaybreaks
\usepackage{color}
\usepackage{mathrsfs}

\begin{document}

\newtheorem{theorem}{Theorem}    
\newtheorem{proposition}[theorem]{Proposition}
\newtheorem{conjecture}[theorem]{Conjecture}
\def\theconjecture{\unskip}
\newtheorem{corollary}[theorem]{Corollary}
\newtheorem{lemma}[theorem]{Lemma}
\newtheorem{sublemma}[theorem]{Sublemma}
\newtheorem{observation}[theorem]{Observation}
\newtheorem{remark}[theorem]{Remark}
\newtheorem{definition}[theorem]{Definition}
\theoremstyle{definition}
\newtheorem{notation}[theorem]{Notation}
\newtheorem{question}[theorem]{Question}
\newtheorem{questions}[theorem]{Questions}
\newtheorem{example}[theorem]{Example}
\newtheorem{problem}[theorem]{Problem}
\newtheorem{exercise}[theorem]{Exercise}

\numberwithin{theorem}{section} \numberwithin{theorem}{section}
\numberwithin{equation}{section}

\def\earrow{{\mathbf e}}
\def\rarrow{{\mathbf r}}
\def\uarrow{{\mathbf u}}
\def\varrow{{\mathbf V}}
\def\tpar{T_{\rm par}}
\def\apar{A_{\rm par}}

\def\reals{{\mathbb R}}
\def\torus{{\mathbb T}}
\def\heis{{\mathbb H}}
\def\integers{{\mathbb Z}}
\def\naturals{{\mathbb N}}
\def\complex{{\mathbb C}\/}
\def\distance{\operatorname{distance}\,}
\def\support{\operatorname{support}\,}
\def\dist{\operatorname{dist}\,}
\def\Span{\operatorname{span}\,}
\def\degree{\operatorname{degree}\,}
\def\kernel{\operatorname{kernel}\,}
\def\dim{\operatorname{dim}\,}
\def\codim{\operatorname{codim}}
\def\trace{\operatorname{trace\,}}
\def\Span{\operatorname{span}\,}
\def\dimension{\operatorname{dimension}\,}
\def\codimension{\operatorname{codimension}\,}
\def\nullspace{\scriptk}
\def\kernel{\operatorname{Ker}}
\def\ZZ{ {\mathbb Z} }
\def\p{\partial}
\def\rp{{ ^{-1} }}
\def\Re{\operatorname{Re\,} }
\def\Im{\operatorname{Im\,} }
\def\ov{\overline}
\def\eps{\varepsilon}
\def\lt{L^2}
\def\diver{\operatorname{div}}
\def\curl{\operatorname{curl}}
\def\etta{\eta}
\newcommand{\norm}[1]{ \|  #1 \|}
\def\expect{\mathbb E}
\def\bull{$\bullet$\ }
\def\C{\mathbb{C}}
\def\R{\mathbb{R}}
\def\Rn{{\mathbb{R}^n}}
\def\Sn{{{S}^{n-1}}}
\def\M{\mathbb{M}}
\def\N{\mathbb{N}}
\def\Q{{\mathbb{Q}}}
\def\Z{\mathbb{Z}}
\def\F{\mathcal{F}}
\def\L{\mathcal{L}}
\def\S{\mathcal{S}}
\def\supp{\operatorname{supp}}
\def\dist{\operatorname{dist}}
\def\essi{\operatornamewithlimits{ess\,inf}}
\def\esss{\operatornamewithlimits{ess\,sup}}
\def\xone{x_1}
\def\xtwo{x_2}
\def\xq{x_2+x_1^2}
\newcommand{\abr}[1]{ \langle  #1 \rangle}

\newcommand{\Norm}[1]{ \left\|  #1 \right\| }
\newcommand{\set}[1]{ \left\{ #1 \right\} }
\def\one{\mathbf 1}
\def\whole{\mathbf V}
\newcommand{\modulo}[2]{[#1]_{#2}}

\def\scriptf{{\mathcal F}}
\def\scriptg{{\mathcal G}}
\def\scriptm{{\mathcal M}}
\def\scriptb{{\mathcal B}}
\def\scriptc{{\mathcal C}}
\def\scriptt{{\mathcal T}}
\def\scripti{{\mathcal I}}
\def\scripte{{\mathcal E}}
\def\scriptv{{\mathcal V}}
\def\scriptw{{\mathcal W}}
\def\scriptu{{\mathcal U}}
\def\scriptS{{\mathcal S}}
\def\scripta{{\mathcal A}}
\def\scriptr{{\mathcal R}}
\def\scripto{{\mathcal O}}
\def\scripth{{\mathcal H}}
\def\scriptd{{\mathcal D}}
\def\scriptl{{\mathcal L}}
\def\scriptn{{\mathcal N}}
\def\scriptp{{\mathcal P}}
\def\scriptk{{\mathcal K}}
\def\frakv{{\mathfrak V}}

%
\newtheorem*{remark0}{\indent\sc Remark}
%
\renewcommand{\proofname}{\indent\sc Proof.} 

\title[Sharp Bounds for the multilinear Hausdorff operator]
{Sharp Bounds for the multilinear Hausdorff operators on mixed radial-angular local Morrey-type spaces}
\author{Ronghui Liu$^{*}$}
\author{Qi Zhang}
\renewcommand{\thefootnote}{}
\footnotetext[1]{*Corresponding author: Ronghui Liu}
\footnotetext[1]{\textcolor{black}{2020} \textit{Mathematics Subject Classification} 26D10. 42B35. 47B38.}

%
\thanks{Ronghui Liu, email: rhliu@nwnu.edu.cn.  Qi Zhang, email: zhangqimaths@126.com}
\thanks{College of Mathematics and Statistics, Northwest Normal University, Lanzhou, 730070, People's Republic
of China}


\maketitle
\begin{abstract}
In this paper, we establish sharp bounds for the multilinear Hausdorff operators on mixed radial-angular local Morrey-type spaces, and we also give ralated applications of these operators. Meanwhile, sharp bounds for the  multilinear Hausdorff operators on mixed radial-angular complementary local Morrey-type spaces are also derived.
\end{abstract}

\section{Introduction}
The Hausdorff operator is associated with the Hausdorff summability method, which arose in the early 20th century in connection with certain classical problems in analysis. Its theory traces its origins to Hurwitz and Silverman \cite{hurwitz1917} in 1917 and was rigorously established by Hausdorff \cite{hausdorff1921} in 1921.

The one-dimensional Hausdorff operator was defined by
$$
h_{\Phi}(f)(x) = \int_{0}^{\infty} \frac{\Phi(t)}{t}f\left(\frac{x}{t}\right)dt,
$$
where $x \in \mathbb{R}$  and $\Phi$ is a locally integrable function on $\mathbb{R}^{+}:=(0,\infty)$.

A general form of the Hausdorff operator in higher dimensions was defined by M¨®ricz in \cite{moricz2005}
\[
\mathcal{H}_{\Phi,A}(f)(x) = \int_{\mathbb{R}^{n}} \frac{\Phi(y)}{|y|^{n}} f(xA(y))dy,
\tag{1.1}\quad x \in \mathbb{R}^n,
\]
where $A(y) = \big( a_{ij}(y) \big)_{i,j=1}^{n}$ is an $n \times n$ matrix whose entries $a_{ij}(y)$ are measurable functions of $y$, and $\det(A(y)) \neq 0$ almost everywhere in the support of $\Phi$.

Now, let
$$
A(y) = \text{diag}(|y|^{-1}, \ldots, |y|^{-1}).
$$

Next, we introduce two different forms of the Hausdorff operators.
One form was introduced by Andersen in \cite{andersen2003},
$$
\mathcal{H}_{\Phi}(f)(x) = \int_{\mathbb{R}^{n}} \frac{\Phi(y)}{|y|^{n}} f\left(\frac{x}{|y|}\right) dy.
$$
The other form was introduced by Chen et al. in \cite{chen2012},
$$
\widetilde{\mathcal{H}_{\Phi}}(f)(x) = \int_{\mathbb{R}^{n}} \frac{\Phi\left(\frac{x}{|y|}\right)}{|y|^{n}} f(y) dy, \quad x \in \mathbb{R}^{n},
$$
where~$\Phi$~is a locally integrable function on $\mathbb{R}^{n}$.

In the past period of time, a particularly notable direction of research is concerned with establishing boundedness properties of Hausdorff operators (see, \cite{andersen2003, chen2012, chen2013, liflyand2013, ruan2016}, etc.).

The modern form of the Hausdorff operator and its variants have been extensively studied, for example, multivariate Hausdorff operators \cite{brown2002, moricz2005, fan2019}, fractional Hausdorff operators \cite{chen2016},  and multidimensional Hausdorff operators \cite{chen2014, lerner2007}, leading to substantial contributions to harmonic analysis.
In addition, many scholars have already established sharp bounds of general Hausdorff operators on some classical function spaces. We refer the reader to \cite{karapetyants2020, wu2015} and the references therein for further details.
Moreover, building upon the work on the multilinear Hardy operator in \cite{fu2012}, Chen et al. \cite{chen2012multilinear} introduced a multilinear extension of the Hausdorff operator and for related subsequent developments, see also \cite{fan2014, gao2015}.

In the sequel, let $m, n_{i} \in \mathbb{N}^{+}$ be positive integers and write
$$
\vec{f} = (f_{1}, f_{2}, \ldots, f_{m}) \quad \text{and} \quad \vec{u} = (u_{1}, u_{2}, \ldots, u_{m})
$$
with each $ u_{i} \in \mathbb{R}^{n_{i}}, i = 1, \ldots, m $. Set
$$
|\vec{u}| = \sqrt{|u_{1}|^{2} + |u_{2}|^{2} + \cdots + |u_{m}|^{2}} \quad \text{and} \quad d\vec{u} = du_{1}du_{2} \cdots du_{m}.
$$

Our main purpose in this paper is to investigate the following four types multilinear extensions of the Hausdorff operator.

The first operator is defined by
$$
R_{\Phi}(\vec{f})(x) = \int_{\mathbb{R}^{n_m}} \cdots \int_{\mathbb{R}^{n_2}} \int_{\mathbb{R}^{n_1}} \frac{\Phi(\vec{u})}{\prod\limits_{i=1}^{m}|u_{i}|^{n_{i}}} \prod_{i=1}^{m}f_{i}\left(\frac{x}{|u_{i}|}\right)d\vec{u}.
$$

The second operator is defined by
$$
\widetilde{R_{\Phi}}(\vec{f})(x) = \int_{\mathbb{R}^{n_m}} \cdots \int_{\mathbb{R}^{n_2}} \int_{\mathbb{R}^{n_1}} \frac{\Phi\left(\frac{x}{|u_{1}|}, \frac{x}{|u_{2}|}, \ldots, \frac{x}{|u_{m}|}\right)}{\prod\limits_{i=1}^{m}|u_{i}|^{n_{i}}} \prod_{i=1}^{m}f_{i}(u_{i})d\vec{u}.
$$

The third operator is defined by
$$
S_{\Phi}(\vec{f})(x) = \int_{\mathbb{R}^{n_m}} \cdots \int_{\mathbb{R}^{n_2}}\int_{\mathbb{R}^{n_1}} \frac{\Phi(\vec{u})}{|\vec{u}|^{\sum\limits_{i=1}^{m} n_i}} \prod_{i=1}^{m} f_i \left( \frac{x}{|u_i|} \right) d\vec{u}.
$$

The last one is defined by
$$
\widetilde{S_{\Phi}}(\vec{f})(x) = \int_{\mathbb{R}^{n_m}} \cdots \int_{\mathbb{R}^{n_2}}\int_{\mathbb{R}^{n_1}} \frac{\Phi\left(\frac{x}{|\vec{u}|}\right)}{|\vec{u}|^{\sum\limits_{i=1}^{m} n_i}} \prod_{i=1}^{m} f_i (u_i) d\vec{u}.
$$

In each of these, $x \in \mathbb{R}^{n} $ and $\Phi$ is a locally integrable function on $\mathbb{R}^{n}$.

One fundamental motivation for the study of Hausdorff operators lies in the fact that many classical operators in analysis arise as special cases of the general Hausdorff operator by making suitable choices of the function~$\Phi$. Therefore, we give some specific examples to demonstrate its wide applications in the third part of this article.

In 2004, Burenkov et al. \cite{burenkov2004} introduced the definition of the local Morrey-type space as follows.
\begin{definition}
Let $0 < p, q \leq \infty$ and $0 \leq \lambda < \infty$. Then the local Morrey-type space $LM^{\lambda}_{p,q} (\mathbb{R}^n)$ for $0 < p < \infty$ and $0 < q < \infty$ is the set of all measurable functions $f$ satisfying
$$
\|f\|_{LM^{\lambda}_{p,q} (\mathbb{R}^n)} = \left( \int_0^\infty \left( \frac{1}{r^{\lambda}} \left( \int_{B(0,r)} |f(y)|^p dy \right)^{\frac{1}{p}} \right)^q \frac{dr}{r} \right)^{\frac{1}{q}} < \infty,
$$
and for $0 < p < \infty$,
$$
\|f\|_{LM^{\lambda}_{p,\infty} (\mathbb{R}^n)} = \sup_{r>0} \frac{1}{r^{\lambda}} \left( \int_{B(0,r)} |f(y)|^p dy \right)^{\frac{1}{p}} < \infty,
$$
where we make the usual modification with~$\|f\|_{L^p(B(0,r))}$~replaced by~$\|f\|_{L^\infty(B(0,r))}$~\\when~$p = \infty$. Here and below, $B(0,r)$ denotes the ball centered at the origin with radius $r > 0$.
\end{definition}

In fact, the space $LM^{\lambda}_{p,\infty}(\mathbb{R}^n)$ coincides with the central Morrey space, as introduced in \cite{alvarez2000}. According to \cite{burenkov2004,guliyev2017}, the space $LM^{\lambda}_{p,q}(\mathbb{R}^n)$ is non-trivial. Specifically, $LM^{\lambda}_{p,q}(\mathbb{R}^n)\neq\Theta$ if and only if
\[
\lambda > 0 \text{ if } q < \infty \quad \text{and } \quad \lambda \geq 0 \text{ if } q = \infty,
\tag{1.3}
\]
where $\Theta$ denotes the set of all functions that are equivalent to zero on $\mathbb{R}^n$.

In 2007, Burenkov et al. \cite{VHV2007} defined the complementary local Morrey-type space as follows.
\begin{definition}
Let~$0 < p, q \leq \infty$~and~$\lambda \leq 0$. Then the complementary local Morrey-type space~${}^{c}LM^{\lambda}_{p,q}(\mathbb{R}^n)$~for~$0 < p < \infty$ and $0<q<\infty$ consists of all measurable functions~$f$~such that
$$
\|f\|_{{}^{c}LM^{\lambda}_{p,q}} = \left( \int_{0}^{\infty} \left( \frac{1}{r^{\lambda}} \left( \int_{{}^{c}B(0,r)} |f(y)|^p dy \right)^{\frac{1}{p}} \right)^q \frac{dr}{r} \right)^{\frac{1}{q}} < \infty,
$$
and for $0 < p < \infty$,
$$
\|f\|_{{}^{c}LM^{\lambda}_{p,\infty}} = \sup_{r>0} \frac{1}{r^{\lambda}} \left( \int_{{}^{c}B(0,r)} |f(y)|^p dy \right)^{\frac{1}{p}} < \infty,
$$
where we make the usual modification with~$\|f\|_{L^p({}^{c}B(0,r))}$~replaced by~$\|f\|_{L^{\infty}({}^{c}B(0,r))}$\\when~$p = \infty$, and ${}^{c}B(0,r)=\mathbb{R}^n \setminus B(0,r)$.
\end{definition}

In addition, the mixed radial-angular space $L^{\tilde{p}}_{\rm rad}L^{p}_{\rm ang}(\mathbb R^n)$, as a extension of classical Lebesgue space $L^{p}(\mathbb R^n)$, was defined as follows.

Let $S^{n-1}$ be the unit sphere in $\mathbb R^n$, $n\geq 2$, with the Lebesgue measure $d\sigma=d\sigma(\cdot)$,
$$
\|f\|_{L^{\tilde{p}}_{\rm rad}L^{p}_{\rm ang}(\mathbb R^n)}:=\left(\int_{0}^{\infty}\|f(r\cdot)\|^{\tilde{p}}_{L^{p}(S^{n-1})}r^{n-1}dr\right)^{\frac{1}{\tilde{p}}},\quad  0< p,
\tilde{p}\leq\infty.
$$
and when $p=\infty$ or $p'=\infty$, we just need to make the usual modifications in the above definition, that is,
$$
\|f\|_{L^{\tilde{p}}_{\rm rad}L^{\infty}_{\rm ang}(\mathbb R^n)}:=\left(\int_{0}^{\infty}\left(\operatorname*{ess\,sup}_{\theta\in\mathbb{S}^{n-1}}|f(\rho\theta)|\right)^{\tilde{p}} r^{n-1}dr\right)^{\frac{1}{\tilde{p}}},
\quad  0< \tilde{p}<\infty,\; p=\infty,
$$
and
$$
\|f\|_{L^{\infty}_{\rm rad}L^{p}_{\rm ang}(\mathbb R^n)}:=\operatorname*{ess\,sup}_{r>0}\|f(r\cdot)\|_{L^{p}(S^{n-1})},
\quad  0< p\leq\infty,\; \tilde{p}=\infty.
$$

The boundedness for some classical operators in harmonic analysis on mixed radial-angular spaces were investigated successively in \cite{CL,DL1,DL2,LF,LLW,LW,RSH2023,RYS2024,RSH2025}.

Naturally, we introduce the mixed radial-angular local Morrey-type space and the mixed radial-angular complementary local Morrey-type space, moreover, the corresponding versions with power weight are given in current work.
\begin{definition}
Let $0 < p, \tilde{p}, q \leq \infty$ and $0 \leq \lambda < \infty$. Then the mixed radial-angular local Morrey-type space $LML_{rad}^{\tilde{p},\lambda,q}L_{ang}^{p}(\mathbb{R}^n, |x|^{\alpha})$ for $0 < p, \tilde{p}, q< \infty$ is the set of all measurable functions $f$ satisfying
\begin{align*}
&\|f\|_{LML_{rad}^{\tilde{p},\lambda,q}L_{ang}^{p}(\mathbb{R}^n, |x|^{\alpha})} \\
&= \left(\int_0^\infty\left(\frac{1}{r^\lambda}\left( \int_0^r\left(\int_{\mathbb{S}^{n-1}} |f(\rho\theta)|^p d\sigma(\theta)\right)^{\frac{\tilde{p}}{p}} \rho^{n-1}\rho^\alpha d\rho \right)^{\frac{1}{\tilde{p}}}\right)^q \frac{dr}{r} \right)^{\frac{1}{q}} < \infty,
\end{align*}
and for $0<p, \tilde{p}<\infty$, $q=\infty$,
\begin{align*}
&\|f\|_{LML_{rad}^{\tilde{p},\lambda,\infty}L_{ang}^{p}(\mathbb{R}^n, |x|^{\alpha})}\\
&= \sup_{r>0} \frac{1}{r^{\lambda}}\left( \int_0^r\left(\int_{\mathbb{S}^{n-1}} |f(\rho\theta)|^p d\sigma(\theta)\right)^{\frac{\tilde{p}}{p}} \rho^{n-1}\rho^\alpha d\rho \right)^{\frac{1}{\tilde{p}}} < \infty,
\end{align*}
and for $0<p, q<\infty$, $\tilde{p}=\infty$,
\begin{align*}
&\|f\|_{LML_{rad}^{\infty,\lambda,q}L_{ang}^{p}(\mathbb{R}^n, |x|^{\alpha})} \\
& = \left( \int_0^\infty r^{-q\lambda} \left(\operatorname*{ess\,sup}_{0<\rho<r} \left( \int_{\mathbb{S}^{n-1}}\left|f(\rho\theta) \right|^p \rho^{\alpha} d\sigma(\theta)\right)^{\frac{1}{p}}\right)^q \frac{dr}{r} \right)^{\frac{1}{q}}< \infty,
\end{align*}
and for $0<\tilde{p}, q<\infty$, $p=\infty$,
\begin{align*}
&\|f\|_{LML_{rad}^{\tilde{p},\lambda,q}L_{ang}^{\infty}(\mathbb{R}^n, |x|^{\alpha})} \\
& = \left( \int_0^\infty r^{-q\lambda} \left(\int_0^r\left(\operatorname*{ess\,sup}_{\theta\in\mathbb{S}^{n-1}}|f(\rho\theta)|\right)^p \rho^{n+\alpha-1} d\rho\right)^{\frac{q}{p}} \frac{dr}{r} \right)^{\frac{1}{q}}< \infty.
\end{align*}
\end{definition}

Note that $\tilde{p} = p = \infty$, $\|f\|_{LML_{rad}^{\infty,\lambda,q}L_{ang}^{\infty}(\mathbb{R}^n, |x|^{\alpha})}=\|f\|_{LM^{\lambda}_{\infty,q} (\mathbb{R}^n)}$. In the same way, $LML_{rad}^{\tilde{p},\lambda,q}L_{ang}^{p}(\mathbb{R}^n, |x|^{\alpha})\neq\Theta$~if and only if~$q$~and~$\lambda$~satisfy (1.3). The space ${LM^{\lambda}_{\infty,q} (\mathbb{R}^n)}$~has already been studied in \cite{An2023}, so this article does not consider this situation.
\begin{definition}
Let $0 < p, \tilde{p}, q \leq \infty$ and $\lambda \leq \infty$. Then the mixed radial-angular complementary local Morrey-type space ${}^{c}LML_{rad}^{\tilde{p},\lambda,q}L_{ang}^{p}(\mathbb{R}^n, |x|^{\alpha})$ for  $0 < p, \tilde{p}, q< \infty$ is the set of all measurable functions $f$ satisfying
\begin{align*}
&\|f\|_{{}^{c}LML_{rad}^{\tilde{p},\lambda,q}L_{ang}^{p}(\mathbb{R}^n, |x|^{\alpha})} \\
&= \left(\int_0^\infty\left(\frac{1}{r^\lambda}\left( \int_r^{\infty}\left(\int_{\mathbb{S}^{n-1}} |f(\rho\theta)|^p d\sigma(\theta)\right)^{\frac{\tilde{p}}{p}} \rho^{n-1}\rho^\alpha d\rho \right)^{\frac{1}{\tilde{p}}}\right)^q \frac{dr}{r} \right)^{\frac{1}{q}} < \infty,
\end{align*}
and for $0<p, \tilde{p}<\infty$, $q=\infty$,
\begin{align*}
&\|f\|_{{}^{c}LML_{rad}^{\tilde{p},\lambda,\infty}L_{ang}^{p}(\mathbb{R}^n, |x|^{\alpha})}\\
&= \sup_{r>0} \frac{1}{r^{\lambda}}\left( \int_r^{\infty}\left(\int_{\mathbb{S}^{n-1}} |f(\rho\theta)|^p d\sigma(\theta)\right)^{\frac{\tilde{p}}{p}} \rho^{n-1}\rho^\alpha d\rho \right)^{\frac{1}{\tilde{p}}} < \infty,
\end{align*}
and for $0<p, q<\infty$, $\tilde{p}=\infty$,
\begin{align*}
&\|f\|_{LML_{rad}^{\infty,\lambda,q}L_{ang}^{p}(\mathbb{R}^n, |x|^{\alpha})} \\
& = \left( \int_0^\infty r^{-q\lambda} \left(\operatorname*{ess\,sup}_{r<\rho<\infty} \left( \int_{\mathbb{S}^{n-1}}\left|f(\rho\theta) \right|^p \rho^{\alpha} d\sigma(\theta)\right)^{\frac{1}{p}}\right)^q \frac{dr}{r} \right)^{\frac{1}{q}}< \infty,
\end{align*}
and for $0<\tilde{p}, q<\infty$, $p=\infty$,
\begin{align*}
&\|f\|_{LML_{rad}^{\tilde{p},\lambda,q}L_{ang}^{\infty}(\mathbb{R}^n, |x|^{\alpha})} \\
& = \left( \int_0^\infty r^{-q\lambda} \left(\int_r^\infty\left(\operatorname*{ess\,sup}_{\theta\in\mathbb{S}^{n-1}}|f(\rho\theta)|\right)^p \rho^{n+\alpha-1} d\rho\right)^{\frac{q}{p}} \frac{dr}{r} \right)^{\frac{1}{q}}< \infty.
\end{align*}
\end{definition}

Recently, An et al. \cite{An2023} obtained necessary and sufficient conditions for the boundedness of multilinear Hausdorff operators on local Morrey-type spaces and sharp contants. In 2025, Wei et al. \cite{Wei2025} gave sharp estimates for the multilinear Hausdorff operators on complementary local Morrey-type spaces. Inspired by the results in \cite{An2023,Wei2025} and the ideas of distinguishing the mixed radial-angular integrabilities. It is natural to consider the corresponding sharp bounds for the multilinear Hausdorff operators on mixed radial-angular local Morrey-type spaces and on mixed radial-angular complementary local Morrey-type spaces. Furthermore, we also establish the results with the power weight.

The outline of this paper is as follows. In Section 2, we state the main results of the multilinear Hausdorff operators on mixed radial-angular local Morrey-type spaces,  some related applications of the main results are given in Section 3. In Section 4, we obtain several supplementary conclusions of the multilinear Hausdorff operators on mixed radial-angular complementary local Morrey-type spaces.

\section{Main results}
We now present our main theorems.

\begin{theorem}\label{th1}
Suppose that~$\Phi$~is a nonnegative, locally integrable, radial function.\\
(i) Let~$1 \leq p, \tilde{p}, q \leq \infty$ and $\alpha\in\mathbb{R}$. If (1.3) be satisfied. Then~$\widetilde{\mathcal{H}_{\Phi}}$~is bounded on the local Morrey-type space~$LML_{rad}^{\tilde{p},\lambda,q}L_{ang}^{p}(\mathbb{R}^n, |x|^{\alpha})$, that is,
$$
\|\widetilde{\mathcal{H}_{\Phi}}(f)\|_{LML_{rad}^{\tilde{p},\lambda,q}L_{ang}^{p}(\mathbb{R}^n, |x|^{\alpha})} \leq C_{\Phi,1}\|f\|_{LML_{rad}^{\tilde{p},\lambda,q}L_{ang}^{p}(\mathbb{R}^n, |x|^{\alpha})}
$$
if and only if
$$
C_{\Phi,1} = \omega_n \int_{0}^{\infty} \frac{\Phi(t)}{t^{\lambda-\frac{n+\alpha}{\tilde{p}} + 1}} dt < \infty.
$$
Moreover,
$$
\|\widetilde{\mathcal{H}_{\Phi}}\|_{LML_{rad}^{\tilde{p},\lambda,q}L_{ang}^{p}(\mathbb{R}^n, |x|^{\alpha}) \to LML_{rad}^{\tilde{p},\lambda,q}L_{ang}^{p}(\mathbb{R}^n, |x|^{\alpha})} = C_{\Phi,1}.
$$
(ii) Let~$0 < p,\tilde{p} < 1$, $0 < q < 1$ and $\alpha\in\mathbb{R}$. If~$f$~is nonnegative. Then
$$
\|\widetilde{\mathcal{H}_{\Phi}}(f)\|_{LML_{rad}^{\tilde{p},\lambda,q}L_{ang}^{p}(\mathbb{R}^n, |x|^{\alpha})} \geq C_{\Phi,1}\|f\|_{LML_{rad}^{\tilde{p},\lambda,q}L_{ang}^{p}(\mathbb{R}^n, |x|^{\alpha})}.
$$
\end{theorem}
\noindent Here and below, $\omega_n$ is the area of the unit sphere in $\mathbb{R}^n$.

\begin{theorem}\label{th2}
Suppose that~$\Phi$~is a nonnegative locally integrable function.\\
(i) Let~$1 \leq p, \tilde{p},q \leq \infty$ and $\alpha\in\mathbb{R}$. If (1.3) be satisfied. Then~$\mathcal{H}_{\Phi}$~is bounded on the local Morrey-type space~$LML_{rad}^{\tilde{p},\lambda,q}L_{ang}^{p}(\mathbb{R}^n, |x|^{\alpha})$, that is,
$$
\|\mathcal{H}_{\Phi}(f)\|_{LML_{rad}^{\tilde{p},\lambda,q}L_{ang}^{p}(\mathbb{R}^n, |x|^{\alpha})} \leq C_{\Phi,2}\|f\|_{LML_{rad}^{\tilde{p},\lambda,q}L_{ang}^{p}(\mathbb{R}^n, |x|^{\alpha})}
$$
if and only if
$$
C_{\Phi,2} = \int_{\mathbb{R}^n} \frac{\Phi(y)}{|y|^{n - \frac{n+\alpha}{\tilde{p}} + \lambda}} dy < \infty.
$$
Moreover,
$$
\|\mathcal{H}_{\Phi}\|_{LML_{rad}^{\tilde{p},\lambda,q}L_{ang}^{p}(\mathbb{R}^n, |x|^{\alpha}) \to LML_{rad}^{\tilde{p},\lambda,q}L_{ang}^{p}(\mathbb{R}^n, |x|^{\alpha})} = C_{\Phi,2}.
$$
(ii) Let~$0 < p, \tilde{p} < 1$, $0 < q < 1$ and $\alpha\in\mathbb{R}$. If~$f$~is nonnegative. Then
$$
\|\mathcal{H}_{\Phi}(f)\|_{LML_{rad}^{\tilde{p},\lambda,q}L_{ang}^{p}(\mathbb{R}^n, |x|^{\alpha})} \geq C_{\Phi,2}\|f\|_{LML_{rad}^{\tilde{p},\lambda,q}L_{ang}^{p}(\mathbb{R}^n, |x|^{\alpha})}.
$$
\end{theorem}

\begin{theorem}\label{th3}
Suppose that~$\Phi$~is a nonnegative locally integrable function.\\
(i) Let~$1 \leq p, p_i, \tilde{p}, \tilde{p_i} < \infty$, $1 \leq q \leq \infty$ satisfy $\frac{1}{p} = \sum\limits_{i=1}^{m} \frac{1}{p_i}$ and $\frac{1}{\tilde{p}} = \sum\limits_{i=1}^{m} \frac{1}{\tilde{p_i}}$. Assume that $\alpha\in\mathbb{R}$, $\lambda = \sum\limits_{i=1}^{m} \lambda_i$~for~$i = 1, \ldots, m$, $\lambda$ satisfies (1.3), and $\lambda_i$~satisfies
\[
\lambda_i > 0 \text{ if } q < \infty \quad \text{and } \quad \lambda_i \geq 0 \text{ if } q = \infty.
\tag{2.1}
\]
If
$$
C_{\Phi,3} = \int_{\mathbb{R}^{n_m}} \cdots \int_{\mathbb{R}^{n_2}} \int_{\mathbb{R}^{n_1}} \frac{\Phi(\vec{u})}{\prod\limits_{i=1}^{m} |u_i|^{n_i}} \prod\limits_{i=1}^{m} |u_i|^{\frac{n+\alpha}{\tilde{p_i}} - \lambda_i} d\vec{u} < \infty,
$$
then
$$
\|R_{\Phi}(\vec{f})\|_{LML_{rad}^{\tilde{p},\lambda,q}L_{ang}^{p}(\mathbb{R}^n, |x|^{\alpha})} \leq C_{\Phi,3} \prod_{i=1}^{m} \|f_i\|_{LML_{rad}^{\tilde{p_i},\lambda,{(q\tilde{p_i})}/\tilde{p}}L_{ang}^{p_i}(\mathbb{R}^n, |x|^{\alpha})}.
$$
Moreover, if
$$
\lambda p = \lambda_i p_i, \quad i = 1, \ldots, m,
$$
then
$$
\|R_{\Phi}\|_{\prod\limits_{i=1}^{m} LML_{rad}^{\tilde{p_i},\lambda,{(q\tilde{p_i})}/\tilde{p}}L_{ang}^{p_i}(\mathbb{R}^n, |x|^{\alpha}) \to LML_{rad}^{\tilde{p},\lambda,q}L_{ang}^{p}(\mathbb{R}^n, |x|^{\alpha})} = C_{\Phi,3}.
$$
(ii) Let~$1 \leq p, p_i, q, q_i < \infty$ satisfy $\frac{1}{p} = \sum\limits_{i=1}^{m} \frac{1}{p_i}$~and~$\frac{1}{q} = \sum\limits_{i=1}^m \frac{1}{q_i}$. Assume that $\alpha\in\mathbb{R}$, $\lambda, \lambda_i > 0$, and $\lambda = \sum\limits_{i=1}^{m} \lambda_i$~for~$i = 1, \ldots, m$. If
$$
\widetilde{C_{\Phi,3}} = \int_{\mathbb{R}^{n_m}} \cdots \int_{\mathbb{R}^{n_2}} \int_{\mathbb{R}^{n_1}} \frac{\Phi(\vec{u})}{\prod\limits_{i=1}^{m} |u_i|^{n_i + \lambda_i -\frac{\alpha}{p_i}}} d\vec{u} < \infty,
$$
then
$$
\|R_{\Phi}(\vec{f})\|_{LML_{rad}^{\infty,\lambda,q}L_{ang}^{p}(\mathbb{R}^n, |x|^{\alpha})} \leq \widetilde{C_{\Phi,3}} \prod\limits_{i=1}^{m} \|f_i\|_{LML_{rad}^{\infty,\lambda_i,q_i}L_{ang}^{p_i}(\mathbb{R}^n, |x|^{\alpha})}.
$$
Moreover, if
$$
\lambda q = \lambda_i q_i, \quad i = 1, \ldots, m,
$$
then
$$
\|R_{\Phi}\|_{\prod\limits_{i=1}^{m} LML_{rad}^{\infty,\lambda_i,q_i}L_{ang}^{p_i}(\mathbb{R}^n, |x|^{\alpha}) \to LML_{rad}^{\infty,\lambda,q}L_{ang}^{p}(\mathbb{R}^n, |x|^{\alpha})} = \widetilde{C_{\Phi,3}}.
$$
\end{theorem}

\begin{theorem}\label{th4}
Suppose that~$\Phi$~is a nonnegative, locally integrable, radial function.
(i) Let~$1 \leq p,  p_i, \tilde{p}, \tilde{p_i} < \infty$, $1 \leq q \leq \infty$ satisfy $\frac{1}{p} = \sum\limits_{i=1}^{m} \frac{1}{p_i}$ and $\frac{1}{\tilde{p}} = \sum\limits_{i=1}^{m} \frac{1}{\tilde{p_i}}$. Assume that $\alpha\in\mathbb{R}$, $\lambda = \sum\limits_{i=1}^m \lambda_i$~for~$i = 1, \ldots, m$, $\lambda$ satisfies (1.3), and $\lambda_i$ satisfies (2.1). If
$$
C_{\Phi,4} = \prod_{i=1}^m \omega_{n_i} \int_0^\infty \cdots \int_0^\infty \frac{\Phi(\vec{t})}{\prod\limits_{i=1}^m t_i^{\lambda_i - \frac{n+\alpha}{\tilde{p_i}} + 1}} d\vec{t} < \infty,
$$
then
$$
\|\widetilde{R_\Phi}(\vec{f})\|_{LML_{rad}^{\tilde{p},\lambda,q}L_{ang}^{p}(\mathbb{R}^n, |x|^{\alpha})} \leq C_{\Phi,4} \prod_{i=1}^m \|f_i\|_{LML_{rad}^{\tilde{p_i},\lambda,{(q\tilde{p_i})}/\tilde{p}}L_{ang}^{p_i}(\mathbb{R}^n, |x|^{\alpha})}.
$$
Moreover, if
$$
\lambda p = \lambda_i p_i, \quad i = 1, \ldots, m,
$$
then
$$
\|\widetilde{R_\Phi}\|_{\prod\limits_{i=1}^m LML_{rad}^{\tilde{p_i},\lambda,{(q\tilde{p_i})}/\tilde{p}}L_{ang}^{p_i}(\mathbb{R}^n, |x|^{\alpha}) \to LML_{rad}^{\tilde{p},\lambda,q}L_{ang}^{p}(\mathbb{R}^n, |x|^{\alpha})} = C_{\Phi,4}.
$$
(ii) Let~$1 \leq p, p_i, q, q_i < \infty$ satisfy $\frac{1}{p} = \sum\limits_{i=1}^{m} \frac{1}{p_i}$~and~$\frac{1}{q} = \sum\limits_{i=1}^m \frac{1}{q_i}$. Assume that $\alpha\in\mathbb{R}$, $\lambda, \lambda_i > 0$, and $\lambda = \sum\limits_{i=1}^m \lambda_i$~for~$i = 1, \ldots, m$. If
$$
\widetilde{C_{\Phi,4}} = \prod_{i=1}^m \omega_{n_i} \int_0^\infty \cdots \int_0^\infty \frac{\Phi(\vec{t})}{\prod\limits_{i=1}^m t_i^{\lambda_i -\frac{\alpha}{p_i} + 1}} d\vec{t} < \infty,
$$
then
$$
\|\widetilde{R_\Phi}(\vec{f})\|_{LML_{rad}^{\infty,\lambda,q}L_{ang}^{p}(\mathbb{R}^n, |x|^{\alpha})} \leq \widetilde{C_{\Phi,4}} \prod_{i=1}^m \|f_i\|_{LML_{rad}^{\infty,\lambda_i,q_i}L_{ang}^{p_i}(\mathbb{R}^n, |x|^{\alpha})}.
$$
Moreover, if
$$
\lambda q = \lambda_i q_i, \quad i = 1, \ldots, m,
$$
then
$$
\|\widetilde{R_\Phi}\|_{\prod\limits_{i=1}^m LML_{rad}^{\infty,\lambda_i,q_i}L_{ang}^{p_i}(\mathbb{R}^n, |x|^{\alpha}) \to LML_{rad}^{\infty,\lambda,q}L_{ang}^{p}(\mathbb{R}^n, |x|^{\alpha})} = \widetilde{C_{\Phi,4}}.
$$
\end{theorem}

\begin{theorem}\label{th5}
Suppose that~$\Phi$~is a nonnegative locally integrable function.\\
(i) Let~$1 \leq  p, p_i, \tilde{p}, \tilde{p_i} < \infty$, $1 \leq q \leq \infty$ satisfy $\frac{1}{p} = \sum\limits_{i=1}^m \frac{1}{p_i}$ and $\frac{1}{\tilde{p}} = \sum\limits_{i=1}^{m} \frac{1}{\tilde{p_i}}$. Assume that $\alpha\in\mathbb{R}$, $\lambda = \sum\limits_{i=1}^m \lambda_i$~ for~$i = 1, \ldots, m$, $\lambda$ satisfies (1.3), and $\lambda_i$ satisfies (2.1). If
$$
C_{\Phi,5} = \int_{\mathbb{R}^{n_m}} \cdots \int_{\mathbb{R}^{n_2}}\int_{\mathbb{R}^{n_1}} \frac{\Phi(\vec{u})}{|\vec{u}|^{\sum\limits_{i=1}^m n_i}} \prod\limits_{i=1}^m |u_i|^{\frac{n+\alpha}{\tilde{p_i}} - \lambda_i} d\vec{u} < \infty,
$$
then
$$
\|S_\Phi(\vec{f})\|_{LML_{rad}^{\tilde{p},\lambda,q}L_{ang}^{p}(\mathbb{R}^n, |x|^{\alpha})} \leq C_{\Phi,5} \prod_{i=1}^m \|f_i\|_{LML_{rad}^{\tilde{p_i},\lambda,{(q\tilde{p_i})}/\tilde{p}}L_{ang}^{p_i}(\mathbb{R}^n, |x|^{\alpha})}.
$$
Moreover, if
$$
\lambda p = \lambda_i p_i, \quad i = 1, \ldots, m,
$$
then
$$
\|S_\Phi\|_{\prod\limits_{i=1}^m LML_{rad}^{\tilde{p_i},\lambda,{(q\tilde{p_i})}/\tilde{p}}L_{ang}^{p_i}(\mathbb{R}^n, |x|^{\alpha}) \to LML_{rad}^{\tilde{p},\lambda,q}L_{ang}^{p}(\mathbb{R}^n, |x|^{\alpha})} = C_{\Phi,5}.
$$
(ii) Let~$1 \leq p, p_i, q, q_i < \infty$ satisfy $\frac{1}{p} = \sum\limits_{i=1}^{m} \frac{1}{p_i}$~and~$\frac{1}{q} = \sum\limits_{i=1}^m \frac{1}{q_i}$. Assume that $\alpha\in\mathbb{R}$, $\lambda, \lambda_i > 0$, and $\lambda = \sum\limits_{i=1}^m \lambda_i$~for~$i = 1, \ldots, m$. If
$$
\widetilde{C_{\Phi,5}} = \int_{\mathbb{R}^{n_m}} \cdots \int_{\mathbb{R}^{n_2}}\int_{\mathbb{R}^{n_1}} \frac{\Phi(\vec{u})}{|\vec{u}|^{\sum\limits_{i=1}^m n_i}} \prod_{i=1}^m |u_i|^{\frac{\alpha}{p_i}-\lambda_i} d\vec{u} < \infty,
$$
then
$$
\|S_\Phi(\vec{f})\|_{LML_{rad}^{\infty,\lambda,q}L_{ang}^{p}(\mathbb{R}^n, |x|^{\alpha})} \leq \widetilde{C_{\Phi,5}} \prod_{i=1}^m \|f_i\|_{LML_{rad}^{\infty,\lambda_i,q_i}L_{ang}^{p_i}(\mathbb{R}^n, |x|^{\alpha})}.
$$
Moreover, if
$$
\lambda q = \lambda_i q_i, \quad i = 1, \ldots, m,
$$
then
$$
\|S_\Phi\|_{\prod\limits_{i=1}^m LML_{rad}^{\infty,\lambda_i,q_i}L_{ang}^{p_i}(\mathbb{R}^n, |x|^{\alpha}) \to LML_{rad}^{\infty,\lambda,q}L_{ang}^{p}(\mathbb{R}^n, |x|^{\alpha})} = \widetilde{C_{\Phi,5}}.
$$
\end{theorem}

\begin{theorem}\label{th6}
Suppose that~$\Phi$ is a nonnegative, locally integrable, radial function.
(i) Let $1 \leq p, p_i,\tilde{p}, \tilde{p_i} < \infty$, $1 \leq q \leq \infty$ satisfy $\frac{1}{p} = \sum\limits_{i=1}^m \frac{1}{p_i}$ and $\frac{1}{\tilde{p}} = \sum\limits_{i=1}^{m} \frac{1}{\tilde{p_i}}$. Assume that $\alpha\in\mathbb{R}$, $\lambda = \sum\limits_{i=1}^m \lambda_i$ for $i = 1, \ldots, m$, $\lambda$ satisfies (1.3), and $\lambda_i$ satisfies (2.1). If
$$
C_{\Phi,6} = \prod_{i=1}^m \omega_{n_i} \int_0^\infty \cdots \int_0^\infty \Phi \left( \frac{\prod\limits_{i=1}^m t_i}{\sqrt{\sum\limits_{i=1}^m t_i^2}} \right) \prod_{i=1}^m \frac{t_i^{\sum\limits_{j=1,j\neq i}^m n_j + \frac{n+\alpha}{\tilde{p_i}}-\lambda_i - 1}}{(\sum\limits_{i=1}^m t_i^2)^{(\sum\limits_{i=1}^m n_i)/2}} d\vec{t} < \infty,
$$
then
$$
\|\widetilde{S_\Phi}(\vec{f})\|_{LML_{rad}^{\tilde{p},\lambda,q}L_{ang}^{p}(\mathbb{R}^n, |x|^{\alpha})} \leq C_{\Phi,6} \prod_{i=1}^m \|f_i\|_{LML_{rad}^{\tilde{p_i},\lambda,{(q\tilde{p_i})}/\tilde{p}}L_{ang}^{p_i}(\mathbb{R}^n, |x|^{\alpha})}.
$$
Moreover, if
$$
\lambda p = \lambda_i p_i, \quad i = 1, \ldots, m,
$$
then
$$
\|\widetilde{S_\Phi}\|_{\prod\limits_{i=1}^m LML_{rad}^{\tilde{p_i},\lambda,{(q\tilde{p_i})}/\tilde{p}}L_{ang}^{p_i}(\mathbb{R}^n, |x|^{\alpha}) \to LML_{rad}^{\tilde{p},\lambda,q}L_{ang}^{p}(\mathbb{R}^n, |x|^{\alpha})} = C_{\Phi,6}.
$$
(ii) Let $1 \leq p, p_i, q, q_i < \infty$ satisfy $\frac{1}{p} = \sum\limits_{i=1}^{m} \frac{1}{p_i}$~and~$\frac{1}{q} = \sum\limits_{i=1}^m \frac{1}{q_i}$. Assume that $\alpha\in\mathbb{R}$, $\lambda, \lambda_i > 0$, and $\lambda = \sum\limits_{i=1}^m \lambda_i$ for $i = 1, \ldots, m$. If
$$
\widetilde{C_{\Phi,6}} = \prod_{i=1}^m \omega_{n_i} \int_0^\infty \cdots \int_0^\infty \Phi \left( \frac{\prod\limits_{i=1}^m t_i}{\sqrt{\sum\limits_{i=1}^m t_i^2}} \right) \prod_{i=1}^m \frac{t_i^{\sum\limits_{j=1,j\neq i}^m n_j + \frac{\alpha}{p_i}-\lambda_i - 1}}{(\sum\limits_{i=1}^m t_i^2)^{(\sum\limits_{i=1}^m n_i)/2}} d\vec{t} < \infty,
$$
then
$$
\|\widetilde{S_\Phi}(\vec{f})\|_{LML_{rad}^{\infty,\lambda,q}L_{ang}^{p}(\mathbb{R}^n, |x|^{\alpha})} \leq \widetilde{C_{\Phi,6}} \prod_{i=1}^m \|f_i\|_{LML_{rad}^{\infty,\lambda_i,q_i}L_{ang}^{p_i}(\mathbb{R}^n, |x|^{\alpha})}.
$$
Moreover, if
$$
\lambda q = \lambda_i q_i, \quad i = 1, \ldots, m,
$$
then
$$
\|\widetilde{S_\Phi}\|_{\prod\limits_{i=1}^m LML_{rad}^{\infty,\lambda_i,q_i}L_{ang}^{p_i}(\mathbb{R}^n, |x|^{\alpha}) \to LML_{rad}^{\infty,\lambda,q}L_{ang}^{p}(\mathbb{R}^n, |x|^{\alpha})} = \widetilde{C_{\Phi,6}}.
$$
\end{theorem}

We now proceed to the proof of the main theorems.

\begin{proof}[Proof of Theorem \ref{th1}]
Noting that~$\Phi$~is a radial function, we now express the Hausdorff operator $\widetilde{\mathcal{H}_{\Phi}}$ in polar coordinates as follows,
$$
\widetilde{\mathcal{H}_{\Phi}}(f)(x) = \int_{\mathbb{R}^n} \frac{\Phi\left(\frac{x}{|y|}\right)}{|y|^n} f(y) dy = \int_0^\infty \int_{\mathbb{S}^{n-1}} f\left(\frac{|x|}{t} y'\right) d\sigma(y') \frac{\Phi(t)}{t} dt.
$$
(i) Assume that~$1 \leq q \leq \infty$~and~$1 \leq p,\tilde{p} < \infty$. For any~$r > 0$, by Minkowski's integral inequality and H\"{o}lder's inequality, we get that
\begin{align*}
&\left( \int_0^r\left(\int_{\mathbb{S}^{n-1}} |\widetilde{\mathcal{H}_{\Phi}}(f)(\rho\theta)|^p d\sigma(\theta)\right)^{\frac{\tilde{p}}{p}} \rho^{n-1}\rho^\alpha d\rho \right)^{\frac{1}{\tilde{p}}} \\
& =\left( \int_0^r\left(\int_{\mathbb{S}^{n-1}} \left|\int_0^\infty \int_{\mathbb{S}^{n-1}} f\left(\frac{|\rho\theta|}{t} y'\right) d\sigma(y') \frac{\Phi(t)}{t} dt\right|^p d\sigma(\theta)\right)^{\frac{\tilde{p}}{p}} \rho^{n-1}\rho^\alpha d\rho \right)^{\frac{1}{\tilde{p}}} \\
& =\omega_n^{\frac{1}{p}}\left( \int_0^r\left|\int_0^\infty \int_{\mathbb{S}^{n-1}} f\left(\frac{\rho}{t} y'\right) d\sigma(y') \frac{\Phi(t)}{t} dt\right|^{\tilde{p}} \rho^{n-1}\rho^\alpha d\rho \right)^{\frac{1}{\tilde{p}}} \\
&\leq \omega_n^{\frac{1}{p}}\int_0^\infty \left( \int_0^r \left(\int_{\mathbb{S}^{n-1}}\left|f\left(\frac{\rho}{t} y'\right)\right| d\sigma(y')\right)^{\tilde{p}}\rho^{n+\alpha-1}d\rho\right)^{\frac{1}{\tilde{p}}}  \frac{\Phi(t)}{t}dt \\
&\leq \omega_n^{\frac{1}{p}}\omega_n^{\frac{1}{p'}}\int_0^\infty \left( \int_0^r \left(\int_{\mathbb{S}^{n-1}}\left|f\left(\frac{\rho}{t} y'\right)\right|^p d\sigma(y')\right)^{\frac{\tilde{p}}{p}}\rho^{n+\alpha-1}d\rho\right)^{\frac{1}{\tilde{p}}}  \frac{\Phi(t)}{t}dt \\
&\leq \omega_n\int_0^\infty t^{\frac{n+\alpha}{\tilde{p}}} \left( \int_0^{\frac{r}{t}} \left(\int_{\mathbb{S}^{n-1}}\left|f\left(sy'\right)\right|^p d\sigma(y')\right)^{\frac{\tilde{p}}{p}}s^{n+\alpha-1}ds\right)^{\frac{1}{\tilde{p}}}  \frac{\Phi(t)}{t}dt.
\end{align*}

Here and below, $\frac{1}{p} + \frac{1}{p'} = 1$~for~$1 < p < \infty$, and~$p=1$, $p' = \infty$.

We now divide~$q$~into two cases: $q = \infty$~and~$1 \leq q < \infty$.

If~$q = \infty$, then for all $\lambda \geq 0$, the following estimate holds,
\begin{align*}
&\frac{1}{r^{\lambda}} \int_0^\infty \left( \int_0^{\frac{r}{t}} \left(\int_{\mathbb{S}^{n-1}}\left|f\left(sy'\right)\right|^p d\sigma(y')\right)^{\frac{\tilde{p}}{p}}s^{n+\alpha-1}ds\right)^{\frac{1}{\tilde{p}}}  \frac{\Phi(t)}{t}t^{\frac{n+\alpha}{\tilde{p}}}dt \\
&=\int_0^\infty \left( \left(\frac{r}{t}\right)^{-\lambda}\int_0^{\frac{r}{t}} \left(\int_{\mathbb{S}^{n-1}}\left|f\left(sy'\right)\right|^p d\sigma(y')\right)^{\frac{\tilde{p}}{p}}s^{n+\alpha-1}ds\right)^{\frac{1}{\tilde{p}}}  \frac{\Phi(t)}{t}t^{\frac{n+\alpha}{\tilde{p}}-\lambda}dt\\
&\leq \int_0^\infty \frac{\Phi(t)}{t^{\lambda-\frac{n+\alpha}{\tilde{p}}+ 1}} dt \|f\|_{LML_{rad}^{\tilde{p},\lambda,\infty}L_{ang}^{p}(\mathbb{R}^n, |x|^{\alpha})}.
\end{align*}

Thus, taking the supremum over all \( r > 0 \), we obtain the following result,
$$
\|\widetilde{\mathcal{H}_{\Phi}}(f)\|_{LML_{rad}^{\tilde{p},\lambda,\infty}L_{ang}^{p}(\mathbb{R}^n, |x|^{\alpha})} \leq \omega_n \int_0^\infty \frac{\Phi(t)}{t^{\lambda-\frac{n+\alpha}{\tilde{p}}+ 1}} dt \|f\|_{LML_{rad}^{\tilde{p},\lambda,\infty}L_{ang}^{p}(\mathbb{R}^n, |x|^{\alpha})}.
$$

If~$1 \leq q < \infty$, by the Minkowski integral inequality, for any~$\lambda > 0$, we have
\begin{align*}
&\|\widetilde{\mathcal{H}_{\Phi}}(f)\|_{LML_{rad}^{\tilde{p},\lambda,q}L_{ang}^{p}(\mathbb{R}^n, |x|^{\alpha})} \\
& = \left(\int_0^\infty\left(\frac{1}{r^\lambda}\left( \int_0^r\left(\int_{\mathbb{S}^{n-1}} |\widetilde{\mathcal{H}_{\Phi}}(f)(\rho\theta)|^p d\sigma(\theta)\right)^{\frac{\tilde{p}}{p}} \rho^{n-1}\rho^\alpha d\rho \right)^{\frac{1}{\tilde{p}}}\right)^q \frac{dr}{r} \right)^\frac{1}{q}\\
& \leq \omega_n \left(\int_0^\infty\left(\frac{1}{r^{\lambda}} \int_0^\infty \left( \int_0^{\frac{r}{t}} \left(\int_{\mathbb{S}^{n-1}}\left|f\left(sy'\right)\right|^p d\sigma(y')\right)^{\frac{\tilde{p}}{p}}s^{n+\alpha-1}ds\right)^{\frac{1}{\tilde{p}}}  \frac{\Phi(t)}{t}t^{\frac{n+\alpha}{\tilde{p}}}dt \right)^q \frac{dr}{r} \right)^\frac{1}{q}\\
& \leq \omega_n \int_0^\infty\left( \int_0^\infty r^{-q\lambda}\left( \int_0^{\frac{r}{t}} \left(\int_{\mathbb{S}^{n-1}}\left|f\left(sy'\right)\right|^p d\sigma(y')\right)^{\frac{\tilde{p}}{p}}s^{n+\alpha-1}ds\right)^{\frac{q}{\tilde{p}}} \frac{dr}{r} \right)^\frac{1}{q} \frac{\Phi(t)}{t}t^{\frac{n+\alpha}{\tilde{p}}}dt \\
&\leq\omega_n \int_0^\infty \frac{\Phi(t)}{t^{\lambda-\frac{n+\alpha}{\tilde{p}}+ 1}} dt \|f\|_{LML_{rad}^{\tilde{p},\lambda,q}L_{ang}^{p}(\mathbb{R}^n, |x|^{\alpha})}.
\end{align*}

If~$1 \leq q \leq \infty$, $1 \leq p < \infty$~and~$\tilde{p} = \infty$, then it is trivial to replace~$\|f\|_{L_{rad}^{\tilde{p}}L_{ang}^p(B(0,r))}$~\\with~$\|f\|_{L_{rad}^{\infty}L_{ang}^p(B(0,r))}$.

We now establish that the constant $C_{\Phi,2}$ is sharp. Proving this requires that the analysis be split into two cases: $1 \leq q < \infty$ and $q = \infty$.

Case 1: $1 \leq q < \infty$. Let~$1 \leq p,\tilde{p}< \infty$. For a given $\lambda > 0$ and sufficiently small $\epsilon > 0$, we select a real number $\beta$, parameterized by $\epsilon$ (to be specified later), subject to the constraint $0 < \beta + \frac{n+\alpha}{\tilde{p}} < \lambda$. Take
\begin{align*}
f_{0}(x) = |x|^{\beta} \chi_{\{|x|>1\}}(x).\tag{2.2}
\end{align*}
Here and below, $\chi_{E}$ is the characteristic function of a measurable set $E$.

The following estimate can therefore be established,
\begin{align*}
&\|f_{0}\|^q_{LML_{rad}^{\tilde{p},\lambda,q}L_{ang}^{p}(\mathbb{R}^n, |x|^{\alpha})} \\
&= \int_0^\infty r^{-q\lambda - 1} \left( \int_0^r\left(\int_{\mathbb{S}^{n-1}} |f_{0}(\rho\theta)|^p d\sigma(\theta)\right)^{\frac{\tilde{p}}{p}} \rho^{n-1}\rho^\alpha d\rho \right)^{\frac{q}{\tilde{p}}}  dr \\
&= \omega_n^{\frac{q}{p}}\int_1^\infty r^{-q\lambda - 1} \left( \int_1^r \rho^{\beta\tilde{p}+n+\alpha-1} d\rho \right)^{\frac{q}{\tilde{p}}} dr \\
&= \omega_n^{\frac{q}{p}}\left(\frac{1}{\beta\tilde{p}+n+\alpha}\right)^{\frac{q}{\tilde{p}}}\int_1^\infty r^{-q\lambda - 1} \left(r^{\beta\tilde{p}+n+\alpha}-1 \right)^{\frac{q}{\tilde{p}}} dr \\
&= \omega_n^{\frac{q}{p}}\left(\frac{1}{\beta\tilde{p}+n+\alpha}\right)^{\frac{q}{\tilde{p}}}\frac{1}{\beta\tilde{p}+n+\alpha}\int_0^\infty (r+1)^{-\frac{q\lambda}{\beta\tilde{p}+n+\alpha} - 1} r^{\frac{q}{\tilde{p}}}dr \\
&=\omega_n^{\frac{q}{p}}\left(\frac{1}{\beta\tilde{p}+n+\alpha}\right)^{\frac{q}{\tilde{p}}+1}B\left(\frac{q}{\tilde{p}}+1,\frac{q\lambda}{\beta\tilde{p}+n+\alpha}-\frac{q}{\tilde{p}}\right),
\end{align*}
where the standard integral definition of the Beta function has been employed,
$$
B(a,b) = \int_0^1 t^{a-1} (1-t)^{b-1} dt = \int_0^\infty \frac{t^{a-1}}{(1+t)^{a+b}} dt,
$$
for~$a,b > 0$. This implies that~$f_{0} \in LML_{rad}^{\tilde{p},\lambda,q}L_{ang}^{p}(\mathbb{R}^n, |x|^{\alpha})$. We also can get
$$
\widetilde{\mathcal{H}_{\Phi}}(f_{0})(x) = \int_0^\infty \int_{\mathbb{S}^{n-1}} f_{0} \left( \frac{|x|}{t} y' \right) d\sigma(y') \frac{\Phi(t)}{t} dt = \omega_n |x|^{\beta} \int_0^{|x|} \Phi(t) t^{-\beta - 1} dt.
$$

The following estimate holds,
\begin{align*}
&\|\widetilde{\mathcal{H}_{\Phi}}(f_{0})\|^q_{LML_{rad}^{\tilde{p},\lambda,q}L_{ang}^{p}(\mathbb{R}^n, |x|^{\alpha})} \\
& = \int_0^\infty\left(\frac{1}{r^\lambda}\left( \int_0^r\left(\int_{\mathbb{S}^{n-1}} |\widetilde{\mathcal{H}_{\Phi}}(f_0)(\rho\theta)|^p d\sigma(\theta)\right)^{\frac{\tilde{p}}{p}} \rho^{n-1}\rho^\alpha d\rho \right)^{\frac{1}{\tilde{p}}}\right)^q \frac{dr}{r} \\
& =\int_0^\infty r^{-q\lambda-1}\left( \int_0^r\left(\int_{\mathbb{S}^{n-1}} \left|\omega_n |\rho\theta|^{\beta} \int_0^{|\rho\theta|} \Phi(t) t^{-\beta - 1} dt\right|^p d\sigma(\theta)\right)^{\frac{\tilde{p}}{p}} \rho^{n+\alpha-1} d\rho \right)^{\frac{q}{\tilde{p}}} dr \\
& =\omega_n^q \omega_n^{\frac{q}{p}}\int_0^\infty r^{-q\lambda-1}\left( \int_0^r\left( \rho^{\beta} \int_0^{\rho} \Phi(t) t^{-\beta - 1} dt \right)^{\tilde{p}} \rho^{n+\alpha-1} d\rho \right)^{\frac{q}{\tilde{p}}} dr \\
& \geq\omega_n^q \omega_n^{\frac{q}{p}}\int_{\epsilon^{-1}}^\infty r^{-q\lambda-1}\left( \int_{\epsilon^{-1}}^r \rho^{\beta\tilde{p}+n+\alpha-1} \left( \int_0^{\epsilon^{-1}} \Phi(t) t^{-\beta - 1} dt \right)^{\tilde{p}}d\rho \right)^{\frac{q}{\tilde{p}}} dr\\
& = \omega_n^{\frac{q}{p}}\left( \omega_n\int_0^{\epsilon^{-1}} \Phi(t) t^{-\beta - 1} dt \right)^q \int_{\epsilon^{-1}}^\infty r^{-q\lambda-1}\left( \int_{\epsilon^{-1}}^r \rho^{\beta\tilde{p}+n+\alpha-1} d\rho \right)^{\frac{q}{\tilde{p}}} dr.
\end{align*}

Therefore, we can derive
\begin{align*}
&\int_{\epsilon^{-1}}^\infty r^{-q\lambda-1}\left( \int_{\epsilon^{-1}}^r \rho^{\beta\tilde{p}+n+\alpha-1} d\rho \right)^{\frac{q}{\tilde{p}}} dr \\
&=\left(\frac{1}{\beta\tilde{p}+n+\alpha}\right)^{\frac{q}{\tilde{p}}}\int_{\epsilon^{-1}}^\infty r^{-q\lambda-1}\left( r^{\beta\tilde{p}+n+\alpha}-\epsilon^{-(\beta\tilde{p}+n+\alpha)} \right)^{\frac{q}{\tilde{p}}} dr \\
&=\left(\frac{1}{\beta\tilde{p}+n+\alpha}\right)^{\frac{q}{\tilde{p}}}\int_1^\infty t^{-q\lambda-1}\epsilon^{q\lambda+1}\left(\epsilon^{-(\beta\tilde{p}+n+\alpha)} \left(t^{\beta\tilde{p}+n+\alpha}-1\right) \right)^{\frac{q}{\tilde{p}}} dt \\
&=\left(\frac{1}{\beta\tilde{p}+n+\alpha}\right)^{\frac{q}{\tilde{p}}}\epsilon^{-(\beta+\frac{n+\alpha}{\tilde{p}}-\lambda)q} \int_1^\infty t^{-q\lambda-1} \left(t^{\beta\tilde{p}+n+\alpha}-1\right)^{\frac{q}{\tilde{p}}} dt \\
&=\epsilon^{-(\beta+\frac{n+\alpha}{\tilde{p}}-\lambda)q} \left(\frac{1}{\beta\tilde{p}+n+\alpha}\right)^{\frac{q}{\tilde{p}}+1} B\left(\frac{q}{\tilde{p}}+1,\frac{q\lambda}{\beta\tilde{p}+n+\alpha}-\frac{q}{\tilde{p}}\right),
\end{align*}
this leads to the ensuing estimate,
\begin{align*}
&\|\widetilde{\mathcal{H}_{\Phi}}(f_{0})\|^q_{LML_{rad}^{\tilde{p},\lambda,q}L_{ang}^{p}(\mathbb{R}^n, |x|^{\alpha})} \\
& \geq  \omega_n^{\frac{q}{p}}\left( \omega_n\int_0^{\epsilon^{-1}} \Phi(t) t^{-\beta - 1} dt \right)^q \int_{\epsilon^{-1}}^\infty r^{-q\lambda-1}\left( \int_{\epsilon^{-1}}^r \rho^{\beta\tilde{p}+n+\alpha-1} d\rho \right)^{\frac{q}{\tilde{p}}} dr\\
& = \omega_n^{\frac{q}{p}}\left( \omega_n\int_0^{\epsilon^{-1}} \Phi(t) t^{-\beta - 1} dt \right)^q \epsilon^{-(\beta+\frac{n+\alpha}{\tilde{p}}-\lambda)q} \left(\frac{1}{\beta\tilde{p}+n+\alpha}\right)^{\frac{q}{\tilde{p}}+1} \\
&\quad\times B\left(\frac{q}{\tilde{p}}+1,\frac{q\lambda}{m\tilde{p}+n+\alpha}-\frac{q}{\tilde{p}}\right)\\
&= \epsilon^{-(\beta+\frac{n+\alpha}{\tilde{p}}-\lambda)q}\left( \omega_n\int_0^{\epsilon^{-1}} \Phi(t) t^{-\beta - 1} dt \right)^q \|f_{0}\|^q_{LML_{rad}^{\tilde{p},\lambda,q}L_{ang}^{p}(\mathbb{R}^n, |x|^{\alpha})}.
\end{align*}

If~$p = \infty$, the same estimate can be derived by selecting
\begin{align*}
f_{\epsilon}(x) = |x|^{\lambda - \epsilon} \chi_{\{|x|>1\}}(x) \tag{2.3}
\end{align*}
for sufficiently small~$\epsilon > 0$.

Case 2: $q = \infty$. We shall construct different functions according to the cases~$\lambda > 0$~and~$\lambda = 0$.

Let~$\lambda > 0$~and~$1 \leq p < \infty$. We take~$f_{0}$~as in (2.2). By the condition~$0 < \beta + \frac{n+\alpha}{\tilde{p}} < \lambda$, a direct calculation shows that
\begin{align*}
&\|f_{0}\|_{LML_{rad}^{\tilde{p},\lambda,\infty}L_{ang}^{p}(\mathbb{R}^n, |x|^{\alpha})} \\
&= \sup_{r>0} \frac{1}{r^{\lambda}} \left( \int_0^r\left(\int_{\mathbb{S}^{n-1}} |f_0(\rho\theta)|^p d\sigma(\theta)\right)^{\frac{\tilde{p}}{p}} \rho^{n-1}\rho^\alpha d\rho \right)^{\frac{1}{\tilde{p}}} \\
&=\omega_n^{\frac{1}{p}}\sup_{r>1} \frac{1}{r^{\lambda}}\left(\int_1^r \rho^{\beta\tilde{p}+n+\alpha-1} d\rho\right)^{\frac{1}{\tilde{p}}}\\
&= \omega_n^{\frac{1}{p}}\left(\frac{1}{\beta\tilde{p}+n+\alpha}\right)^{\frac{1}{\tilde{p}}} \sup_{r>1} \frac{1}{r^{\lambda}} \left(r^{\beta\tilde{p}+n+\alpha} - 1\right)^{\frac{1}{\tilde{p}}} \\
&= \omega_n^{\frac{1}{p}}\left(\frac{1}{\beta\tilde{p}+n+\alpha}\right)^{\frac{1}{\tilde{p}}} \left( \frac{\beta + \frac{n+\alpha}{\tilde{p}}}{\lambda-(\beta+\frac{n+\alpha}{\tilde{p}})} \right)^{\frac{1}{\tilde{p}}} \left( \frac{\lambda-(\beta+\frac{n+\alpha}{\tilde{p}})}{\lambda} \right)^{\frac{\lambda}{\beta\tilde{p}+n+\alpha}},
\end{align*}
which yields~$f_{0} \in LML_{rad}^{\tilde{p},\lambda,\infty}L_{ang}^{p}(\mathbb{R}^n, |x|^{\alpha})$. Furthermore, we obtain
\begin{align*}
&\|\widetilde{\mathcal{H}_{\Phi}}(f_{0})\|_{LML_{rad}^{\tilde{p},\lambda,\infty}L_{ang}^{p}(\mathbb{R}^n, |x|^{\alpha})} \\
& = \sup_{r>0}\frac{1}{r^\lambda}\left( \int_0^r\left(\int_{\mathbb{S}^{n-1}} |\widetilde{\mathcal{H}_{\Phi}}(f_0)(\rho\theta)|^p d\sigma(\theta)\right)^{\frac{\tilde{p}}{p}} \rho^{n-1}\rho^\alpha d\rho \right)^{\frac{1}{\tilde{p}}} \\
& = \sup_{r>0}\frac{1}{r^\lambda}\left( \int_0^r\left(\int_{\mathbb{S}^{n-1}} \left|\omega_n |\rho\theta|^{\beta} \int_0^{|\rho\theta|} \Phi(t) t^{-\beta - 1} dt\right|^p d\sigma(\theta)\right)^{\frac{\tilde{p}}{p}} \rho^{n+\alpha-1} d\rho \right)^{\frac{1}{\tilde{p}}} \\
& = \omega_n^{1+\frac{1}{p}}\sup_{r>0}\frac{1}{r^\lambda}\left( \int_0^r  \left|\int_0^{\rho} \Phi(t) t^{-\beta - 1} dt \right|^{\tilde{p}} \rho^{\beta\tilde{p}+\alpha+n-1} d\rho \right)^{\frac{1}{\tilde{p}}} \\
& \geq  \omega_n^{1+\frac{1}{p}}\sup_{r>\epsilon^{-1}}\frac{1}{r^\lambda}\left( \int_{\epsilon^{-1}}^r \left|\int_0^{\epsilon^{-1}} \Phi(t) t^{-\beta - 1} dt \right|^{\tilde{p}} \rho^{\beta\tilde{p}+\alpha+n-1} d\rho \right)^{\frac{1}{\tilde{p}}} \\
& =  \omega_n^{1+\frac{1}{p}} \int_0^{\epsilon^{-1}} \Phi(t) t^{-\beta - 1} dt \sup_{r > \epsilon^{-1}} \frac{1}{r^{\lambda}} \left( \int_{\epsilon^{-1}}^r \rho^{\beta\tilde{p}+\alpha+n-1} d\rho \right)^{\frac{1}{\tilde{p}}}\\
& =\omega_n^{1+\frac{1}{p}} \left(\frac{1}{\beta\tilde{p}+\alpha+n}\right)^{\frac{1}{\tilde{p}}} \int_0^{\epsilon^{-1}} \Phi(t) t^{-\beta - 1} dt\sup_{r > \epsilon^{-1}} \frac{1}{r^{\lambda}} \left( r^{\beta\tilde{p}+\alpha+n}-\epsilon^{-(\beta\tilde{p}+\alpha+n)} \right)^{\frac{1}{\tilde{p}}}\\
& =\epsilon^{\lambda-\beta-\frac{\alpha+n}{\tilde{p}}}\omega_n^{1+\frac{1}{p}} \left(\frac{1}{\beta\tilde{p}+\alpha+n}\right)^{\frac{1}{\tilde{p}}} \int_0^{\epsilon^{-1}} \Phi(t) t^{-\beta - 1} dt\sup_{t >1} \frac{1}{t^{\lambda}} \left( t^{\beta\tilde{p}+\alpha+n}-1 \right)^{\frac{1}{\tilde{p}}}\\
& = \epsilon^{\lambda-\beta-\frac{\alpha+n}{\tilde{p}}}\omega_n \int_0^{\epsilon^{-1}} \Phi(t) t^{-\beta - 1} dt\|f_{\epsilon}\|_{LML_{rad}^{\tilde{p},\lambda,\infty}L_{ang}^{p}(\mathbb{R}^n, |x|^{\alpha})}.
\end{align*}

For~$\lambda > 0$~and~$\tilde{p} = \infty$, we can take~$f_{\epsilon}$~as in (2.3).

For the case~$\lambda = 0$, notice that~$LML_{rad}^{\tilde{p},0,\infty}L_{ang}^{p} = L_{rad}^{\tilde{p}}L_{ang}^{p}$~when~$1 \leq p < \infty$, and~$LML_{rad}^{\infty,0,\infty}L_{ang}^{\infty} = L^{\infty}$. For sufficiently small ~$\epsilon > 0$, we take
$$
f_{\epsilon}(x) = |x|^{-\frac{n}{\tilde{p}} - \epsilon} \chi_{\{|x|>1\}}(x) \quad \text{for } 1 \leq p,\tilde{p} < \infty,
$$
and
$$
f_{\epsilon}(x) = |x|^{-\epsilon} \chi_{\{|x|>1\}}(x) \quad \text{for } \tilde{p} = \infty,
$$
respectively.

Collecting the results from Cases 1 and 2, we define~$\beta = \lambda - \frac{n+\alpha}{\tilde{p}} - \epsilon$~for~$1 \leq \tilde{p} \leq \infty$~to arrive at the estimate,
\begin{align*}
&\|\widetilde{\mathcal{H}_{\Phi}}\|_{LML_{rad}^{\tilde{p},\lambda,q}L_{ang}^{p}(\mathbb{R}^n, |x|^{\alpha}) \to LML_{rad}^{\tilde{p},\lambda,q}L_{ang}^{p}(\mathbb{R}^n, |x|^{\alpha}) } \\
&\geq \frac{\|\widetilde{\mathcal{H}_{\Phi}}(f_{0})\|_{LML_{rad}^{\tilde{p},\lambda,q}L_{ang}^{p}(\mathbb{R}^n, |x|^{\alpha})}}{\|f_{0}\|_{LML_{rad}^{\tilde{p},\lambda,q}L_{ang}^{p}(\mathbb{R}^n, |x|^{\alpha})}} \\
&\geq \epsilon^{\epsilon} \omega_n \int_0^{\epsilon^{-1}} \Phi(t)t^{-\beta-1} dt.
\end{align*}

Letting~$\epsilon \to 0^+$, we obtain the desired result.

(ii) Assume that~$0 < p < 1$~and~$0 < q < 1$. For any~$r > 0$, by Minkowski's integral inequality and H\"{o}lder's inequality for~$0 < p < 1$, we deduce that
\begin{align*}
&\left( \int_0^r\left(\int_{\mathbb{S}^{n-1}} |\widetilde{\mathcal{H}_{\Phi}}(f)(\rho\theta)|^p d\sigma(\theta)\right)^{\frac{\tilde{p}}{p}} \rho^{n-1}\rho^\alpha d\rho \right)^{\frac{1}{\tilde{p}}} \\
&=\left( \int_0^r\left(\int_{\mathbb{S}^{n-1}} \left(\int_0^\infty \int_{\mathbb{S}^{n-1}} f\left(\frac{|\rho\theta|}{t} y'\right) d\sigma(y') \frac{\Phi(t)}{t} dt\right)^p d\sigma(\theta)\right)^{\frac{\tilde{p}}{p}} \rho^{n-1}\rho^\alpha d\rho \right)^{\frac{1}{\tilde{p}}} \\
&=\omega_n^{\frac{1}{p}}\left( \int_0^r \left(\int_0^\infty \int_{\mathbb{S}^{n-1}} f\left(\frac{\rho}{t} y'\right) d\sigma(y') \frac{\Phi(t)}{t} dt\right)^{\tilde{p}}\rho^{n+\alpha-1}d\rho \right)^{\frac{1}{\tilde{p}}} \\
&\geq\omega_n^{\frac{1}{p}}\int_0^\infty \left(\int_0^r \left( \int_{\mathbb{S}^{n-1}} f\left(\frac{\rho}{t} y'\right) d\sigma(y') \right)^{\tilde{p}}\rho^{n+\alpha-1}d\rho \right)^{\frac{1}{\tilde{p}}}\frac{\Phi(t)}{t} dt\\
&=\omega_n^{\frac{1}{p}}\int_0^\infty \left(\int_0^{\frac{r}{t}} \left( \int_{\mathbb{S}^{n-1}} f\left(\tilde{\rho}y'\right) d\sigma(y') \right)^{\tilde{p}}\tilde{\rho}^{n+\alpha-1}d\tilde{\rho} \right)^{\frac{1}{\tilde{p}}}\frac{\Phi(t)}{t} t^{\frac{n+\alpha}{\tilde{p}}}dt \qquad (\tilde{\rho}=\frac{\rho}{t})\\
&\geq\omega_n^{\frac{1}{p}}\omega_n^{\frac{1}{p'}}\int_0^\infty \left(\int_0^{\frac{r}{t}} \left( \int_{\mathbb{S}^{n-1}}  \left(f\left(\tilde{\rho}y'\right)\right)^p d\sigma(y') \right)^{\frac{\tilde{p}}{p}}\tilde{\rho}^{n+\alpha-1}d\tilde{\rho} \right)^{\frac{1}{\tilde{p}}}\frac{\Phi(t)}{t} t^{\frac{n+\alpha}{\tilde{p}}}dt.
\end{align*}

Since~$0 < q < 1$, using Minkowski's integral inequality again,  we arrive at
\begin{align*}
&\|\widetilde{\mathcal{H}_{\Phi}}(f)\|_{LML_{rad}^{\tilde{p},\lambda,q}L_{ang}^{p}(\mathbb{R}^n, |x|^{\alpha})} \\
&\geq \omega_n\left(\int_0^\infty r^{-q\lambda}\left(\int_0^\infty \left(\int_0^{\frac{r}{t}} \left( \int_{\mathbb{S}^{n-1}}  \left(f\left(\tilde{\rho}y'\right)\right)^p d\sigma(y') \right)^{\frac{\tilde{p}}{p}}\tilde{\rho}^{n+\alpha-1}d\tilde{\rho} \right)^{\frac{1}{\tilde{p}}}\frac{\Phi(t)}{t} t^{\frac{n+\alpha}{\tilde{p}}}dt \right)^q \frac{dr}{r} \right)^{\frac{1}{q}}\\
&\geq \omega_n\int_0^\infty\left(\int_0^\infty r^{-q\lambda} \left(\int_0^{\frac{r}{t}} \left( \int_{\mathbb{S}^{n-1}}  \left(f\left(\tilde{\rho}y'\right)\right)^p d\sigma(y') \right)^{\frac{\tilde{p}}{p}}\tilde{\rho}^{n+\alpha-1}d\tilde{\rho} \right)^{\frac{q}{\tilde{p}}} \frac{dr}{r}\right)^{\frac{1}{q}} \frac{\Phi(t)}{t}t^{\frac{n+\alpha}{\tilde{p}}}dt \\
&=\omega_n \int_0^\infty \frac{\Phi(t)}{t^{\lambda-\frac{n+\alpha}{\tilde{p}}+ 1}} dt \|f\|_{LML_{rad}^{\tilde{p},\lambda,q}L_{ang}^{p}(\mathbb{R}^n, |x|^{\alpha})}.
\end{align*}

The proof of Theorem \ref{th1} is completed.
\end{proof}

\begin{proof}[Proof of Theorem \ref{th2}]
The proof is similar to that of Theorem \ref{th1}. If~$\lambda > 0$, then for sufficiently small~$\epsilon > 0$, we take the radial function
$$
f_{\epsilon}(x) = |x|^{\lambda - \frac{n}{\tilde{p}} - \epsilon} \chi_{\{|x| > 1\}}(x) \quad \text{for } 1 \leq p,\tilde{p} < \infty,
$$
and
$$
f_{\epsilon}(x) = |x|^{\lambda - \epsilon} \chi_{\{|x| > 1\}}(x) \quad \text{for } \tilde{p} = \infty,
$$
to establish the sharp bound. If~$\lambda = 0$, the function defined in Theorem \ref{th1} can be adopted here without change.
\end{proof}

In the subsequent analysis concerning multilinear Hausdorff operators, we restrict our attention to the bilinear case (\(m = 2\)) for clarity of exposition. The arguments presented extend naturally to the general multilinear setting \(m > 2\).

\begin{proof}[Proof of Theorem \ref{th3}]
(i) For any~$r > 0$, by Minkowski's integral inequality and H\"{o}lder's inequality with~$\frac{1}{p} = \frac{1}{p_1} + \frac{1}{p_2}$~and~$\frac{1}{\tilde{p}} = \frac{1}{\tilde{p_1}} + \frac{1}{\tilde{p_2}}$, we get that
\begin{align*}
&\left( \int_0^r\left(\int_{\mathbb{S}^{n-1}} |R_{\Phi}(f_1, f_2)(\rho\theta)|^p d\sigma(\theta)\right)^{\frac{\tilde{p}}{p}}\rho^{n-1}\rho^{\alpha} d\rho\right)^{\frac{1}{\tilde{p}}}\\
&= \left( \int_0^r\left(\int_{\mathbb{S}^{n-1}}\left| \int_{\mathbb{R}^{n_2}} \int_{\mathbb{R}^{n_1}} \frac{\Phi(u_1, u_2)}{|u_1|^{n_1} |u_2|^{n_2}} \prod_{i=1}^2 f_i \left( \frac{\rho\theta}{|u_i|} \right) du_1 du_2 \right|^p d\sigma(\theta)\right)^{\frac{\tilde{p}}{p}}\rho^{n+\alpha-1} d\rho\right)^{\frac{1}{\tilde{p}}}\\
&\leq \int_{\mathbb{R}^{n_2}} \int_{\mathbb{R}^{n_1}}\left( \int_0^r\left(\int_{\mathbb{S}^{n-1}}\left| \prod_{i=1}^2 f_i \left( \frac{\rho\theta}{|u_i|} \right) \right|^p d\sigma(\theta)\right)^{\frac{\tilde{p}}{p}}\rho^{n+\alpha-1} d\rho\right)^{\frac{1}{\tilde{p}}} \frac{\Phi(u_1, u_2)}{|u_1|^{n_1} |u_2|^{n_2}} du_1 du_2 \\
&\leq \int_{\mathbb{R}^{n_2}} \int_{\mathbb{R}^{n_1}} \prod_{i=1}^2 \left( \int_0^r\left(\int_{\mathbb{S}^{n-1}}\left|  f_i \left( \frac{\rho\theta}{|u_i|} \right) \right|^{p_i} d\sigma(\theta)\right)^{\frac{\tilde{p_i}}{p_i}}\rho^{n+\alpha-1} d\rho\right)^{\frac{1}{\tilde{p_i}}} \frac{\Phi(u_1, u_2)}{|u_1|^{n_1} |u_2|^{n_2}} du_1 du_2 \\
&= \int_{\mathbb{R}^{n_2}} \int_{\mathbb{R}^{n_1}} \prod_{i=1}^2 \left( \int_0^{\frac{r}{|u_i|}}\left(\int_{\mathbb{S}^{n-1}}\left|  f_i \left(y_i\theta \right) \right|^{p_i} d\sigma(\theta)\right)^{\frac{\tilde{p_i}}{p_i}}{y_i}^{n+\alpha-1} d{y_i}\right)^{\frac{1}{\tilde{p_i}}}  \frac{\Phi(u_1, u_2)}{|u_1|^{n_1} |u_2|^{n_2}}\\
&\quad\times|u_1|^{\frac{n+\alpha}{\tilde{p_1}}}|u_2|^{\frac{n+\alpha}{\tilde{p_2}}} du_1 du_2.
\end{align*}

If $q = \infty$, then we make the following estimates,
\begin{align*}
&\frac{1}{r^{\lambda}} \int_{\mathbb{R}^{n_2}} \int_{\mathbb{R}^{n_1}} \prod_{i=1}^2 \left( \int_0^{\frac{r}{|u_i|}}\left(\int_{\mathbb{S}^{n-1}}\left|  f_i \left(y_i\theta \right) \right|^{p_i} d\sigma(\theta)\right)^{\frac{\tilde{p_i}}{p_i}}{y_i}^{n+\alpha-1} d{y_i}\right)^{\frac{1}{\tilde{p_i}}}  \frac{\Phi(u_1, u_2)}{|u_1|^{n_1} |u_2|^{n_2}} \\
&\quad\times|u_1|^{\frac{n+\alpha}{\tilde{p_1}}}|u_2|^{\frac{n+\alpha}{\tilde{p_2}}} du_1 du_2\\
& = \int_{\mathbb{R}^{n_2}} \int_{\mathbb{R}^{n_1}} \prod_{i=1}^2 \frac{1}{(\frac{r}{u_i})^{\lambda_i}} \left( \int_0^{\frac{r}{|u_i|}}\left(\int_{\mathbb{S}^{n-1}}\left|  f_i \left(y_i\theta \right) \right|^{p_i} d\sigma(\theta)\right)^{\frac{\tilde{p_i}}{p_i}}{y_i}^{n+\alpha-1} d{y_i}\right)^{\frac{1}{\tilde{p_i}}}  \frac{\Phi(u_1, u_2)}{|u_1|^{n_1} |u_2|^{n_2}} \\
&\quad\times|u_1|^{\frac{n+\alpha}{\tilde{p_1}}-\lambda_1}|u_2|^{\frac{n+\alpha}{\tilde{p_2}}-\lambda_2} du_1 du_2\\
& \leq \int_{\mathbb{R}^{n_2}} \int_{\mathbb{R}^{n_1}} \frac{\Phi(u_1, u_2)}{|u_1|^{n_1} |u_2|^{n_2}} |u_1|^{\frac{n+\alpha}{\tilde{p_1}} - \lambda_1} |u_2|^{\frac{n+\alpha}{\tilde{p_2}} - \lambda_2} du_1 du_2 \prod_{i=1}^2 \|f_i\|_{LML_{rad}^{\tilde{p_i},\lambda,\infty}L_{ang}^{p_i}(\mathbb{R}^n, |x|^{\alpha})}.
\end{align*}

Hence, we deduce that
\begin{align*}
&\|R_{\Phi}(f_1, f_2)\|_{LML_{rad}^{\tilde{p},\lambda,\infty}L_{ang}^{p}(\mathbb{R}^n, |x|^{\alpha})} \\
&\leq \int_{\mathbb{R}^{n_2}} \int_{\mathbb{R}^{n_1}} \frac{\Phi(u_1, u_2)}{|u_1|^{n_1} |u_2|^{n_2}} |u_1|^{\frac{n+\alpha}{\tilde{p_1}} - \lambda_1} |u_2|^{\frac{n+\alpha}{\tilde{p_2}} - \lambda_2} du_1 du_2 \prod_{i=1}^2 \|f_i\|_{LML_{rad}^{\tilde{p_i},\lambda,\infty}L_{ang}^{p_i}(\mathbb{R}^n, |x|^{\alpha})}.
\end{align*}

If~$1 \leq q < \infty$, using Minkowski's integral inequality and H\"{o}lder's inequality, we obtain
\begin{align*}
&\|R_{\Phi}(f_1, f_2)\|_{LML_{rad}^{\tilde{p},\lambda,q}L_{ang}^{p}(\mathbb{R}^n, |x|^{\alpha})} \\
&= \left( \int_0^\infty r^{-q\lambda} \left( \int_0^r\left(\int_{\mathbb{S}^{n-1}} |R_{\Phi}(f_1, f_2)(\rho\theta)|^p d\sigma(\theta)\right)^{\frac{\tilde{p}}{p}}\rho^{n-1}\rho^{\alpha} d\rho\right)^{\frac{q}{\tilde{p}}} \frac{dr}{r} \right)^{\frac{1}{q}} \\
& \leq \Bigg( \int_0^\infty r^{-q(\lambda_1 + \lambda_2)} \Bigg( \int_{\mathbb{R}^{n_2}} \int_{\mathbb{R}^{n_1}} \prod_{i=1}^2 \Bigg( \int_0^{\frac{r}{|u_i|}}\left(\int_{\mathbb{S}^{n-1}}\left|  f_i \left(y_i\theta \right) \right|^{p_i} d\sigma(\theta)\right)^{\frac{\tilde{p_i}}{p_i}}{y_i}^{n+\alpha-1} d{y_i}\Bigg)^{\frac{1}{\tilde{p_i}}}\\
&\quad\times \frac{\Phi(u_1, u_2)}{|u_1|^{n_1} |u_2|^{n_2}} |u_1|^{\frac{n+\alpha}{\tilde{p_1}}} |u_2|^{\frac{n+\alpha}{\tilde{p_2}}} du_1 du_2 \Bigg)^{q} \frac{dr}{r} \Bigg)^{\frac{1}{q}} \\
& \leq \int_{\mathbb{R}^{n_2}} \int_{\mathbb{R}^{n_1}} \Bigg(\int_0^\infty\prod_{i=1}^2 r^{-q\lambda_i} \Bigg( \int_0^{\frac{r}{|u_i|}}\left(\int_{\mathbb{S}^{n-1}}\left|  f_i \left(y_i\theta \right) \right|^{p_i} d\sigma(\theta)\right)^{\frac{\tilde{p_i}}{p_i}}{y_i}^{n+\alpha-1} d{y_i}\Bigg)^{\frac{q}{\tilde{p_i}}}\frac{dr}{r} \Bigg)^{\frac{1}{q}} \\
&\quad\times \frac{\Phi(u_1, u_2)}{|u_1|^{n_1} |u_2|^{n_2}} |u_1|^{\frac{n+\alpha}{\tilde{p_1}}} |u_2|^{\frac{n+\alpha}{\tilde{p_2}}} du_1 du_2 \\
& \leq \int_{\mathbb{R}^{n_2}} \int_{\mathbb{R}^{n_1}} \prod_{i=1}^2 \Bigg(\int_0^\infty r^{-\frac{q\tilde{p_i}\lambda_i}{\tilde{p}}} \Bigg[\Bigg( \int_0^{\frac{r}{|u_i|}}\left(\int_{\mathbb{S}^{n-1}}\left|  f_i \left(y_i\theta \right) \right|^{p_i} d\sigma(\theta)\right)^{\frac{\tilde{p_i}}{p_i}}{y_i}^{n+\alpha-1} d{y_i}\Bigg)^{\frac{1}{\tilde{p_i}}} \Bigg]^{\frac{q\tilde{p_i}}{\tilde{p}}} \\
&\quad\frac{dr}{r} \Bigg)^{\frac{\tilde{p}}{q}\frac{1}{\tilde{p_i}}}\times \frac{\Phi(u_1, u_2)}{|u_1|^{n_1} |u_2|^{n_2}} |u_1|^{\frac{n+\alpha}{\tilde{p_1}}} |u_2|^{\frac{n+\alpha}{\tilde{p_2}}} du_1 du_2\\
& \leq \int_{\mathbb{R}^{n_2}} \int_{\mathbb{R}^{n_1}} \prod_{i=1}^2 \Bigg(\int_0^\infty \left(\frac{r}{|u_i|}\right)^{-\frac{q\tilde{p_i}\lambda_i}{\tilde{p}}} \Bigg[\Bigg( \int_0^{\frac{r}{|u_i|}}\left(\int_{\mathbb{S}^{n-1}}\left|  f_i \left(y_i\theta \right) \right|^{p_i} d\sigma(\theta)\right)^{\frac{\tilde{p_i}}{p_i}}{y_i}^{n+\alpha-1} d{y_i}\Bigg)^{\frac{1}{\tilde{p_i}}} \Bigg]^{\frac{q\tilde{p_i}}{\tilde{p}}} \\
&\quad\times\frac{dr}{r} \Bigg)^{\frac{\tilde{p}}{q}\frac{1}{\tilde{p_i}}} \frac{\Phi(u_1, u_2)}{|u_1|^{n_1} |u_2|^{n_2}} |u_1|^{\frac{n+\alpha}{\tilde{p_1}}-\lambda_1} |u_2|^{\frac{n+\alpha}{\tilde{p_2}}-\lambda_2} du_1 du_2\\
& \leq \int_{\mathbb{R}^{n_2}} \int_{\mathbb{R}^{n_1}}\frac{\Phi(u_1, u_2)}{|u_1|^{n_1} |u_2|^{n_2}} |u_1|^{\frac{n+\alpha}{\tilde{p_1}}-\lambda_1} |u_2|^{\frac{n+\alpha}{\tilde{p_2}}-\lambda_2} du_1 du_2 \prod_{i=1}^2\|f_i\|_{LML_{rad}^{\tilde{p_i},\lambda,(q\tilde{p_i})/{\tilde{p}}}L_{ang}^{p_i}(\mathbb{R}^n, |x|^{\alpha})}.
\end{align*}

Next, we demonstrate that the constant $C_{\Phi,3}$ corresponds to the operator norm of $R_{\Phi}$. Following the previous approach, the analysis can be divided into the cases $1 \leq q < \infty$ and $q = \infty$.

We begin by addressing the case~$1 \leq q < \infty$. For an arbitrarily small~$\epsilon > 0$, we select two real numbers~$\beta_1$~and~$\beta_2$, which depend on~$\epsilon$~and will be specified later, such that~$0 < \beta_i + \frac{n+\alpha}{\tilde{p_i}} < \lambda_i$~ for~$i = 1, 2$. For ~$x \in \mathbb{R}^n$, taking
$$
f_{i}(x) = |x|^{\beta_i} \chi_{\{|x|>1\}}(x), \quad i = 1, 2.
$$

By computations similar to the previous ones, it follows that
\begin{align*}
&\|f_{i}\|^{\frac{q\tilde{p_i}}{\tilde{p}}}_{LML_{rad}^{\tilde{p_i},\lambda_i,(q\tilde{p_i})/{\tilde{p}}}L_{ang}^{p_i}(\mathbb{R}^n, |x|^{\alpha})}\\
&= \int_0^\infty  r^{-\frac{q\tilde{p_i}\lambda_i}{\tilde{p}}}\Bigg[\Bigg( \int_0^{r}\left(\int_{\mathbb{S}^{n-1}}\left|  f_i \left(y_i\theta \right) \right|^{p_i} d\sigma(\theta)\right)^{\frac{\tilde{p_i}}{p_i}}{y_i}^{n+\alpha-1} d{y_i}\Bigg)^{\frac{1}{\tilde{p_i}}} \Bigg]^{\frac{q\tilde{p_i}}{\tilde{p}}}\frac{dr}{r}\\
&= \int_1^\infty  r^{-\frac{q\tilde{p_i}\lambda_i}{\tilde{p}}-1}\Bigg( \int_1^{r}\left(\int_{\mathbb{S}^{n-1}}\left|y_i\theta\right|^{\beta_ip_i} d\sigma(\theta)\right)^{\frac{\tilde{p_i}}{p_i}}{y_i}^{n+\alpha-1} d{y_i}\Bigg)^{\frac{q}{\tilde{p}}}dr\\
&=\omega_n^{\frac{q\tilde{p_i}}{p_i\tilde{p}}} \int_1^\infty  r^{-\frac{q\tilde{p_i}\lambda_i}{\tilde{p}}-1}\Bigg( \int_1^{r}\left|y_i\right|^{\beta_i\tilde{p_i}+n+\alpha-1}d{y_i}\Bigg)^{\frac{q}{\tilde{p}}}dr\\
&=\omega_n^{\frac{q\tilde{p_i}}{p_i\tilde{p}}}\left(\frac{1}{\beta_i\tilde{p_i}+n+\alpha}\right)^{\frac{q}{\tilde{p}}}  \int_1^\infty  r^{-\frac{q\tilde{p_i}\lambda_i}{\tilde{p}}-1}\Bigg( r^{\beta_i\tilde{p_i}+n+\alpha}-1\Bigg)^{\frac{q}{\tilde{p}}}dr\\
&=\omega_n^{\frac{q\tilde{p_i}}{p_i\tilde{p}}}\left(\frac{1}{\beta_i\tilde{p_i}+n+\alpha}\right)^{\frac{q}{\tilde{p}}+1}  \int_0^\infty  (r+1)^{-(\frac{q}{\tilde{p}}(\frac{\tilde{p_i}\lambda_i}{\beta_i\tilde{p_i}+n+\alpha})+1)} r^{\frac{q}{\tilde{p}}}dr\\
&=\omega_n^{\frac{q\tilde{p_i}}{p_i\tilde{p}}}\left(\frac{1}{\beta_i\tilde{p_i}+n+\alpha}\right)^{\frac{q}{\tilde{p}}+1}  B\left(\frac{q}{\tilde{p}}+1, \frac{q}{\tilde{p}}(\frac{\tilde{p_i}\lambda_i}{\beta_i\tilde{p_i}+n+\alpha}-1)\right).
\end{align*}

Employing the reasoning established in the proof of Theorem \ref{th2}, we obtain
\begin{align*}
&\|R_{\Phi}(f_{1}, f_{2})\|^q_{LML_{rad}^{\tilde{p},\lambda,q}L_{ang}^{p}(\mathbb{R}^n, |x|^{\alpha})} \\
&=\int_0^\infty r^{-q\lambda} \left( \int_0^r\left(\int_{\mathbb{S}^{n-1}} |R_{\Phi}(f_1, f_2)(\rho\theta)|^p d\sigma(\theta)\right)^{\frac{\tilde{p}}{p}}\rho^{n-1}\rho^{\alpha} d\rho\right)^{\frac{q}{\tilde{p}}} \frac{dr}{r}  \\
&=\int_0^\infty r^{-q\lambda} \Bigg( \int_0^r\Bigg(\int_{\mathbb{S}^{n-1}} \Bigg| \int_{\mathbb{R}^{n_2}} \int_{\mathbb{R}^{n_1}} \frac{\Phi(u_1, u_2)}{|u_1|^{n_1} |u_2|^{n_2}} \prod_{i=1}^2 f_{i} \Bigg( \frac{\rho\theta}{|u_i|} \Bigg) du_1 du_2 \Bigg|^p d\sigma(\theta)\Bigg)^{\frac{\tilde{p}}{p}}\\
&\quad\times\rho^{n-1}\rho^{\alpha} d\rho\Bigg)^{\frac{q}{\tilde{p}}} \frac{dr}{r}  \\
&=\int_0^\infty r^{-q\lambda} \Bigg( \int_0^r\Bigg(\int_{\mathbb{S}^{n-1}} \Bigg| \int_{|u_2|<\rho} \int_{|u_1|<\rho} \frac{\Phi(u_1, u_2)}{|u_1|^{n_1+\beta_1} |u_2|^{n_2+\beta_2}}  du_1 du_2 \Bigg|^p d\sigma(\theta)\Bigg)^{\frac{\tilde{p}}{p}}\\
&\quad\times\rho^{(\beta_1+\beta_2)\tilde{p}+n+\alpha-1} d\rho\Bigg)^{\frac{q}{\tilde{p}}} \frac{dr}{r}  \\
&=\left(\int_{|u_2|<\rho} \int_{|u_1|<\rho} \frac{\Phi(u_1, u_2)}{|u_1|^{n_1+\beta_1} |u_2|^{n_2+\beta_2}}  du_1 du_2 \right)^q \\
&\quad\times\omega_n^{\frac{q}{p}}\int_0^\infty r^{-q\lambda-1}\Bigg(\int_0^r \rho^{(\beta_1+\beta_2)\tilde{p}+n+\alpha-1} d\rho\Bigg)^{\frac{q}{\tilde{p}}}dr\\
&\geq\left(\int_{|u_2|<{\epsilon^{-1}}} \int_{|u_1|<{\epsilon^{-1}}} \frac{\Phi(u_1, u_2)}{|u_1|^{n_1+\beta_1} |u_2|^{n_2+\beta_2}}  du_1 du_2 \right)^q \\
&\quad\times\omega_n^{\frac{q}{p}}\int_{\epsilon^{-1}}^\infty r^{-q\lambda-1}\Bigg(\int_{\epsilon^{-1}}^r \rho^{(\beta_1+\beta_2)\tilde{p}+n+\alpha-1} d\rho\Bigg)^{\frac{q}{\tilde{p}}}dr.
\end{align*}

We now consider the case~$q = \infty$. If~$\lambda_i > 0$~for~$i = 1, 2$, then we get that
\begin{align*}
&\|f_{i}\|_{LML_{rad}^{\tilde{p_i},\lambda_i,\infty}L_{ang}^{p_i}(\mathbb{R}^n, |x|^{\alpha})} \\
& = \sup_{r>0}\frac{1}{r^{\lambda_i}}\left( \int_0^r\left(\int_{\mathbb{S}^{n-1}} |f_{i}(\rho\theta)|^{p_i} d\sigma(\theta)\right)^{\frac{\tilde{p_i}}{p_i}} \rho^{n-1}\rho^\alpha d\rho \right)^{\frac{1}{\tilde{p_i}}} \\
& = \sup_{r>1}\frac{1}{r^{\lambda_i}}\left( \int_1^r\left(\int_{\mathbb{S}^{n-1}} \rho^{\beta_ip_i} d\sigma(\theta)\right)^{\frac{\tilde{p_i}}{p_i}} \rho^{n-1}\rho^\alpha d\rho \right)^{\frac{1}{\tilde{p_i}}} \\
& = \omega_n^{\frac{1}{p_i}}\sup_{r>1}\frac{1}{r^{\lambda_i}}\left( \int_1^r \rho^{\beta_i\tilde{p_i}+n+\alpha-1} d\rho\right)^{\frac{1}{\tilde{p_i}}}\\
& = \omega_n^{\frac{1}{p_i}}\left(\frac{1}{\beta_i\tilde{p_i}+n+\alpha}\right)^{\frac{1}{\tilde{p_i}}} \sup_{r>1}\frac{1}{r^{\lambda_i}}\left( r^{\beta_i\tilde{p_i}+n+\alpha} -1 \right)^{\frac{1}{\tilde{p_i}}},
\end{align*}
and
\begin{align*}
&\|R_{\Phi}(f_{1}, f_{2})\|_{LML_{rad}^{\tilde{p},\lambda,\infty}L_{ang}^{p}(\mathbb{R}^n, |x|^{\alpha})} \\
& = \sup_{r>0} \frac{1}{r^{\lambda}} \Bigg( \int_0^r \Bigg( \int_{\mathbb{S}^{n-1}} \Bigg| \int_{|u_2| < \rho} \int_{|u_1| < \rho} \frac{\Phi(u_1, u_2)}{|u_1|^{n_1 + \beta_1} |u_2|^{n_2 + \beta_2}} du_1 du_2 \Bigg|^p  d\sigma(\theta)\Bigg)^{\frac{\tilde{p}}{p}}\\
&\quad\times\rho^{(\beta_1 + \beta_2)\tilde{p}+n+\alpha-1} d\rho \Bigg)^{\frac{1}{\tilde{p}}} \\
& \geq \omega_n^{\frac{1}{p}}\sup_{r>\epsilon^{-1}} \frac{1}{r^{\lambda}} \Bigg( \int_{\epsilon^{-1}}^r  \Bigg| \int_{|u_2| < {\epsilon^{-1}}} \int_{|u_1| < {\epsilon^{-1}}} \frac{\Phi(u_1, u_2)}{|u_1|^{n_1 + \beta_1} |u_2|^{n_2 + \beta_2}} du_1 du_2 \Bigg|^{\tilde{p}}\\
&\quad\times\rho^{(\beta_1 + \beta_2)\tilde{p}+n+\alpha-1} d\rho \Bigg)^{\frac{1}{\tilde{p}}} \\
& =\omega_n^{\frac{1}{p}}\int_{|u_2| < {\epsilon^{-1}}} \int_{|u_1| < {\epsilon^{-1}}} \frac{\Phi(u_1, u_2)}{|u_1|^{n_1 + \beta_1} |u_2|^{n_2 + \beta_2}} du_1 du_2 \\ &\quad\times\sup_{r>\epsilon^{-1}}\frac{1}{r^{\lambda}}\Bigg( \int_{\epsilon^{-1}}^r \rho^{(\beta_1 + \beta_2)\tilde{p}+n+\alpha-1} d\rho  \Bigg)^{\frac{1}{\tilde{p}}}.
\end{align*}

Letting~$\beta_i = \lambda_i - \frac{n+\alpha}{\tilde{p_i}} - \frac{\epsilon}{\tilde{p_i}}$~for~$i = 1, 2$, we arrive at
\begin{align*}
&\|f_{i}\|_{LML_{rad}^{\tilde{p_i},\lambda_i,(q\tilde{p_i})/{\tilde{p}}}L_{ang}^{p_i}(\mathbb{R}^n, |x|^{\alpha})} \\
&= \omega_n^{\frac{1}{p_i}}\left(\frac{1}{\beta_i\tilde{p}+n+\alpha} \right)^{\frac{1}{\tilde{p_i}}+\frac{\tilde{p}}{q\tilde{p_i}}} \left( B \left( \frac{q}{\tilde{p}} + 1, \frac{q}{\tilde{p}} \left( \frac{\tilde{p_i}\lambda_i}{\beta_i\tilde{p_i}+n+\alpha} -1\right) \right) \right)^{\frac{\tilde{p}}{q\tilde{p_i}}}\\
&= \omega_n^{\frac{1}{p_i}}\left(\frac{1}{\tilde{p_i}\lambda_i-\epsilon} \right)^{\frac{1}{\tilde{p_i}}+\frac{\tilde{p}}{q\tilde{p_i}}} \left( B \left( \frac{q}{\tilde{p}} + 1, \frac{q}{\tilde{p}} \left( \frac{\epsilon}{\tilde{p_i} \lambda_i - \epsilon} \right) \right) \right)^{\frac{\tilde{p}}{q\tilde{p_i}}},
\end{align*}
when~$1 \leq q < \infty$, and
\begin{align*}
&\|f_{i}\|_{LML_{rad}^{\tilde{p_i},\lambda_i,\infty}L_{ang}^{p_i}(\mathbb{R}^n, |x|^{\alpha})} \\
& = \omega_n^{\frac{1}{p_i}}\left(\frac{1}{\beta_i\tilde{p_i}+n+\alpha}\right)^{\frac{1}{\tilde{p_i}}} \sup_{r>1}\frac{1}{r^{\lambda_i}}\left( r^{\beta_i\tilde{p_i}+n+\alpha} -1 \right)^{\frac{1}{\tilde{p_i}}}\\
&= \omega_n^{\frac{1}{p_i}} \left( \frac{1}{\tilde{p_i} \lambda_i - \epsilon} \right)^{\frac{1}{\tilde{p_i}}} \sup_{r>1} \frac{1}{r^{\lambda_i}} (r^{\tilde{p_i} \lambda_i - \epsilon} - 1)^{\frac{1}{\tilde{p_i}}}.
\end{align*}

Observing that $\lambda = \lambda_1 + \lambda_2$, it follows that~$\beta_1 + \beta_2 = \lambda - \frac{n+\alpha}{\tilde{p}} - \frac{\epsilon}{p}$. By means of straightforward algebraic manipulation under the conditions~$p \lambda = p_1 \lambda_1 = p_2 \lambda_2, \quad \frac{1}{\tilde{p}} = \frac{1}{\tilde{p_1}} + \frac{1}{\tilde{p_2}}$ and $\frac{1}{p} = \frac{1}{p_1} + \frac{1}{p_2}$, we arrive at the estimates below.

For~$1 \leq q < \infty$, we obtain
\begin{align*}
&\omega_n^{\frac{1}{p}} \left(\int_{\epsilon^{-1}}^\infty r^{-q\lambda-1}\Bigg(\int_{\epsilon^{-1}}^r \rho^{(\beta_1+\beta_2)\tilde{p}+n+\alpha-1} d\rho\Bigg)^{\frac{q}{\tilde{p}}}dr \right)^{\frac{1}{q}}\\
&= \omega_n^{\frac{1}{p}} \Bigg(\frac{1}{(\beta_1+\beta_2)\tilde{p}+n+\alpha}\Bigg)^{\frac{1}{\tilde{p}}}\Bigg(\int_{\epsilon^{-1}}^\infty r^{-q\lambda-1}\Bigg(r^{(\beta_1+\beta_2)\tilde{p}+n+\alpha}-\epsilon^{-(\beta_1+\beta_2)\tilde{p}-n-\alpha}\Bigg)^{\frac{q}{\tilde{p}}}dr\Bigg)^{\frac{1}{q}}\\
&= \omega_n^{\frac{1}{p}} \Bigg(\frac{1}{(\beta_1+\beta_2)\tilde{p}+n+\alpha}\Bigg)^{\frac{1}{\tilde{p}}} \epsilon^{\lambda-(\beta_1+\beta_2)-\frac{n+\alpha}{\tilde{p}}}\Bigg(\int_1^\infty t^{-q\lambda-1}\Bigg(t^{(\beta_1+\beta_2)\tilde{p}+n+\alpha}-1\Bigg)^{\frac{q}{\tilde{p}}}dt\Bigg)^{\frac{1}{q}}\\
&= \omega_n^{\frac{1}{p}} \epsilon^{\lambda-(\beta_1+\beta_2)-\frac{n+\alpha}{\tilde{p}}}\Bigg(\frac{1}{(\beta_1+\beta_2)\tilde{p}+n+\alpha}\Bigg)^{\frac{1}{\tilde{p}}+\frac{1}{q}}\Bigg( B \Bigg( \frac{q}{\tilde{p}} + 1,\frac{q\lambda}{(\beta_1+\beta_2)\tilde{p}+n+\alpha}-\frac{q}{\tilde{p}}\Bigg)\Bigg)^{\frac{1}{q}} \\
& = \omega_n^{\frac{1}{p}} \epsilon^{\frac{\epsilon}{\tilde{p}}}  \left( \frac{1}{\tilde{p} \lambda - \epsilon} \right)^{\frac{1}{\tilde{p}}+\frac{1}{q}} \left( B \left( \frac{q}{\tilde{p}} + 1, \frac{q}{\tilde{p}} \left( \frac{\epsilon}{\tilde{p} \lambda - \epsilon} \right) \right) \right)^{\frac{1}{q}} \\
&= \epsilon^{\frac{\epsilon}{\tilde{p}}} \prod_{i=1}^2 \|f_{i}\|_{LML_{rad}^{\tilde{p_i},\lambda_i,\infty}L_{ang}^{p_i}(\mathbb{R}^n, |x|^{\alpha})}.
\end{align*}

If~$q = \infty$~and~$\lambda_i > 0$~for~$i = 1, 2$, then we get that
\begin{align*}
&\omega_n^{\frac{1}{p}} \sup_{r>\epsilon^{-1}} \frac{1}{r^{\lambda}} \Bigg(\int_{\epsilon^{-1}}^r \rho^{(\beta_1+\beta_2)\tilde{p}+n+\alpha-1} d\rho\Bigg)^{\frac{1}{\tilde{p}}}\\
&=\omega_n^{\frac{1}{p}} \left( \frac{1}{(\beta_1+\beta_2)\tilde{p}+n+\alpha}\right)^{\frac{1}{\tilde{p}}}\sup_{r>\epsilon^{-1}}\frac{1}{r^{\lambda}} \Bigg(r^{(\beta_1+\beta_2)\tilde{p}+n+\alpha}-\epsilon^{-(\beta_1+\beta_2)\tilde{p}-n-\alpha}  \Bigg)^{\frac{1}{\tilde{p}}} \\
&=\omega_n^{\frac{1}{p}} \left( \frac{1}{(\beta_1+\beta_2)\tilde{p}+n+\alpha}\right)^{\frac{1}{\tilde{p}}}\epsilon^{\lambda-(\beta_1+\beta_2)-\frac{n+\alpha}{\tilde{p}}}\sup_{t>1}\frac{1}{t^{\lambda}}\Bigg(t^{(\beta_1+\beta_2)\tilde{p}+n+\alpha}-1\Bigg)^{\frac{1}{\tilde{p}}}  \\
&= \epsilon^{\frac{\epsilon}{p}} \omega_n^{\frac{1}{p}} \left( \frac{1}{\tilde{p} \lambda - \epsilon} \right)^{\frac{1}{\tilde{p}}} \sup_{t>1} \frac{1}{t^{\lambda}} (t^{\tilde{p} \lambda - \epsilon} - 1)^{\frac{1}{\tilde{p}}} \\
&= \epsilon^{\frac{\epsilon}{p}} \prod_{i=1}^2 \|f_{i}\|_{LML_{rad}^{\tilde{p_i},\lambda_i,\infty}L_{ang}^{p_i}(\mathbb{R}^n, |x|^{\alpha})}.
\end{align*}

If~$q = \infty$~and~$\lambda_i = 0$~for~$i = 1, 2$, the sharp bound is then established by taking
$$
f_{i}^\epsilon(x) = |x|^{-\frac{n}{\tilde{p_i}} - \frac{\epsilon}{\tilde{p_i}}} \chi_{\{|x|>1\}}(x)
$$
for sufficiently small~$\epsilon > 0$~and~$1 \leq p_i < \infty$~with~$i = 1, 2$.

Thus, for~$1 \leq q \leq \infty$, we deduce that
\begin{align*}
&\|R_{\Phi}\|_{\prod\limits_{i=1}^2 LML_{rad}^{\tilde{p_i},\lambda_i,q_i}L_{ang}^{p_i}(\mathbb{R}^n, |x|^{\alpha}) \to LML_{rad}^{\tilde{p},\lambda,q}L_{ang}^{p}(\mathbb{R}^n, |x|^{\alpha})} \\
&\geq \frac{\|R_{\Phi}(f_{1}, f_{2})\|_{ LML_{rad}^{\tilde{p},\lambda,q}L_{ang}^{p}(\mathbb{R}^n, |x|^{\alpha})}}{\prod\limits_{i=1}^2 \|f_{i}\|_{LML_{rad}^{\tilde{p_i},\lambda_i,(q \tilde{p_i})/\tilde{p}}L_{ang}^{p_i}(\mathbb{R}^n, |x|^{\alpha})}} \\
&\geq \epsilon^{\frac{\epsilon}{p}} \int_{|u_2| < \epsilon^{-1}} \int_{|u_1| < \epsilon^{-1}} \frac{\Phi(u_1, u_2)}{|u_1|^{n_1} |u_2|^{n_2}} \prod_{i=1}^2 |u_i|^{\frac{n+\alpha + \epsilon}{\tilde{p_i}} - \lambda_i} du_1 du_2,
\end{align*}
where we need to replace~$f_{i}$~with~$f_{i}^\epsilon$~when~$q = \infty$, and~$\lambda_i = 0$~for~$i = 1, 2$. Letting~$\epsilon \to 0^+$, we get the desired result.

(ii) By Minkowski's integral inequality and H\"{o}lder's inequality with~$\frac{1}{q} = \frac{1}{q_1} + \frac{1}{q_2}$, we have
\begin{align*}
&\|R_{\Phi}(f_1, f_2)\|_{LML_{rad}^{\infty,\lambda,q}L_{ang}^{p}(\mathbb{R}^n, |x|^{\alpha})} \\
& = \left( \int_0^\infty r^{-q\lambda} \left(\operatorname*{ess\,sup}_{0<\rho<r} \left( \int_{\mathbb{S}^{n-1}}\left|R_{\Phi}(f_1, f_2)(\rho\theta) \right|^p \rho^{\alpha} d\sigma(\theta)\right)^{\frac{1}{p}}\right)^q \frac{dr}{r} \right)^{\frac{1}{q}} \\
& = \Bigg( \int_0^\infty r^{-q\lambda} \Bigg(\operatorname*{ess\,sup}_{0<\rho<r} \Bigg( \int_{\mathbb{S}^{n-1}}\Bigg|\int_{\mathbb{R}^{n_2}} \int_{\mathbb{R}^{n_1}} \frac{\Phi(u_1, u_2)}{|u_1|^{n_1} |u_2|^{n_2}} \prod_{i=1}^2 f_i \Bigg( \frac{\rho\theta}{|u_i|} \Bigg) du_1 du_2 \Bigg|^p \\
& \quad\times\rho^{\alpha} d\sigma(\theta)\Bigg)^{\frac{1}{p}}\Bigg)^q \frac{dr}{r} \Bigg)^{\frac{1}{q}} \\
& \leq \int_{\mathbb{R}^{n_2}} \int_{\mathbb{R}^{n_1}} \frac{\Phi(u_1, u_2)}{|u_1|^{n_1} |u_2|^{n_2}} \Bigg( \int_0^\infty r^{-q\lambda}\Bigg( \int_{\mathbb{S}^{n-1}} \prod_{i=1}^2 \operatorname*{ess\,sup}_{0<\rho<r} \Bigg| f_i \Bigg( \frac{\rho\theta}{|u_i|} \Bigg) \Bigg|^p \rho^\alpha d\sigma(\theta)\Bigg)^{\frac{q}{p}}\\
&\quad\times\frac{dr}{r} \Bigg)^{\frac{1}{q}} du_1 du_2 \\
& \leq \int_{\mathbb{R}^{n_2}} \int_{\mathbb{R}^{n_1}} \frac{\Phi(u_1, u_2)}{|u_1|^{n_1} |u_2|^{n_2}} \prod_{i=1}^2 \Bigg( \int_0^\infty r^{-q_i\lambda_i}\Bigg( \int_{\mathbb{S}^{n-1}} \operatorname*{ess\,sup}_{0<\rho<r} \Bigg| f_i \Bigg( \frac{\rho\theta}{|u_i|} \Bigg) \Bigg|^{p_i} \rho^\alpha d\sigma(\theta)\Bigg)^{\frac{q_i}{p_i}}\\
&\quad\times\frac{dr}{r} \Bigg)^{\frac{1}{q_i}} du_1 du_2 \\
& =\int_{\mathbb{R}^{n_2}} \int_{\mathbb{R}^{n_1}} \frac{\Phi(u_1, u_2)}{|u_1|^{n_1} |u_2|^{n_2}} \prod_{i=1}^2 \Bigg( \int_0^\infty r^{-q_i\lambda_i}\Bigg( \int_{\mathbb{S}^{n-1}} \operatorname*{ess\,sup}_{0<\rho<\frac{r}{|u_i|}} \Bigg| f_i \Bigg(y\theta \Bigg) \Bigg|^{p_i} y^\alpha |u_i|^\alpha d\sigma(\theta)\Bigg)^{\frac{q_i}{p_i}}\\
&\quad\times\frac{dr}{r} \Bigg)^{\frac{1}{q_i}} du_1 du_2 \\
& =\int_{\mathbb{R}^{n_2}} \int_{\mathbb{R}^{n_1}} \frac{\Phi(u_1, u_2)}{|u_1|^{n_1+\lambda_1-\frac{\alpha}{p_1}} |u_2|^{n_2+\lambda_2-\frac{\alpha}{p_2}}} \prod_{i=1}^2 \Bigg( \int_0^\infty \tilde{\rho}^{-q_i\lambda_i}\Bigg( \int_{\mathbb{S}^{n-1}} \operatorname*{ess\,sup}_{0<\rho<\tilde{\rho}} \Bigg| f_i \Bigg(y\theta \Bigg) \Bigg|^{p_i}\\
&\quad\times y^\alpha d\sigma(\theta)\Bigg)^{\frac{q_i}{p_i}}\frac{dr}{r} \Bigg)^{\frac{1}{q_i}} du_1 du_2 \\
& = \int_{\mathbb{R}^{n_2}} \int_{\mathbb{R}^{n_1}} \frac{\Phi(u_1, u_2)}{|u_1|^{n_1+\lambda_1-\frac{\alpha}{p_1}} |u_2|^{n_2+\lambda_2-\frac{\alpha}{p_2}}}du_1 du_2 \prod_{i=1}^2 \|f_i\|_{LML_{rad}^{\infty,\lambda_i,q_i}L_{ang}^{p_i}(\mathbb{R}^n, |x|^{\alpha})}.
\end{align*}

To establish the sharp bound, we define the following functions. For an arbitrarily small $\epsilon > 0$, we select two real numbers $\beta_1$ and $\beta_2$, dependent on $\epsilon$ (to be specified later), such that~$0 < \beta_i < \lambda_i$~for~$i = 1, 2$. For~$x \in \mathbb{R}^n$, taking
$$
f_{i}(x) = |x|^{\beta_i} \chi_{\{|x|>1\}}(x), \quad i = 1, 2,
$$
we deduce that
\begin{align*}
&\|f_{i}\|_{LML_{rad}^{\infty,\lambda_i,q_i}L_{ang}^{p_i}(\mathbb{R}^n, |x|^{\alpha})}^{q_i} \\
&= \int_0^\infty r^{-q_i\lambda_i} \left(\operatorname*{ess\,sup}_{0<\rho<r} \left( \int_{\mathbb{S}^{n-1}}\left|f_{\beta_i}(\rho\theta) \right|^{p_i} \rho^{\alpha} d\sigma(\theta)\right)^{\frac{1}{p_i}}\right)^{q_i} \frac{dr}{r}  \\
&=  \int_1^\infty r^{-q_i\lambda_i} \left(\operatorname*{ess\,sup}_{0<\rho<r} \left( \int_{\mathbb{S}^{n-1}}\left|\rho\theta\right|^{\beta_ip_i} \rho^{\alpha} d\sigma(\theta)\right)^{\frac{1}{p_i}}\right)^{q_i} \frac{dr}{r} \\
&=\omega_n^{\frac{q_i}{p_i}}\left( \int_1^\infty r^{-q_i\lambda_i} r^{(\beta_i+\frac{\alpha}{p_i})q_i} \frac{dr}{r} \right)^{\frac{1}{q_i}} \\
&=\omega_n^{\frac{q_i}{p_i}}\frac{1}{(\lambda_i-\beta_i-\frac{\alpha}{p_i})q_i} .
\end{align*}
Then we get
\begin{align*}
&\|R_{\Phi}(f_{1}, f_{2})\|_{LML_{rad}^{\infty,\lambda,q}L_{ang}^{p}(\mathbb{R}^n, |x|^{\alpha})}^q \\
& = \int_0^\infty r^{-q\lambda} \Bigg(\operatorname*{ess\,sup}_{0<\rho<r} \Bigg( \int_{\mathbb{S}^{n-1}}\Bigg|\int_{|u_2|<\rho} \int_{|u_1|<\rho} \frac{\Phi(u_1, u_2)}{|u_1|^{n_1} |u_2|^{n_2}} \prod_{i=1}^2 f_{i} \Bigg( \frac{\rho\theta}{|u_i|} \Bigg) du_1 du_2 \Bigg|^p \\
& \quad\times\rho^{\alpha} d\sigma(\theta)\Bigg)^{\frac{1}{p}}\Bigg)^q \frac{dr}{r} \\
& =  \int_0^\infty r^{-q\lambda} \Bigg(\operatorname*{ess\,sup}_{0<\rho<r} \Bigg( \int_{\mathbb{S}^{n-1}}\Bigg|\int_{|u_2|<\rho} \int_{|u_1|<\rho} \frac{\Phi(u_1, u_2)}{|u_1|^{n_1} |u_2|^{n_2}} \prod_{i=1}^2 \Bigg|\frac{\rho\theta}{|u_i|}\Bigg|^{\beta_i} du_1 du_2 \Bigg|^p \\
& \quad\times\rho^{\alpha} d\sigma(\theta)\Bigg)^{\frac{1}{p}}\Bigg)^q \frac{dr}{r}  \\
&\geq\omega_n^{\frac{q}{p}} \int_{\epsilon^{-1}}^\infty r^{-q\lambda} \operatorname*{ess\,sup}_{\epsilon^{-1}<\rho<r} \rho^{(\beta_1+\beta_2+\frac{\alpha}{p})q}\Bigg(\int_{|u_2|<\epsilon^{-1}} \int_{|u_1|<\epsilon^{-1}} \frac{\Phi(u_1, u_2)}{|u_1|^{n_1+\beta_1} |u_2|^{n_2+\beta_2}}du_1 du_2\Bigg)^q \frac{dr}{r}\\
&=\omega_n^{\frac{q}{p}} \Bigg(\int_{|u_2|<\epsilon^{-1}} \int_{|u_1|<\epsilon^{-1}} \frac{\Phi(u_1, u_2)}{|u_1|^{n_1+\beta_1} |u_2|^{n_2+\beta_2}}du_1 du_2\Bigg)^q \int_{\epsilon^{-1}}^\infty r^{-q\lambda+(\beta_1+\beta_2+\frac{\alpha}{p})q}\frac{dr}{r}.
\end{align*}

We now take~$\beta_i = \lambda_i -\frac{\alpha}{p_i}- \frac{\epsilon}{q_i}$~for~$i = 1, 2$. Then
$$
\|f_{i}\|_{LML_{rad}^{\infty,\lambda_i,q_i}L_{ang}^{p_i}(\mathbb{R}^n, |x|^{\alpha})}^{q_i} = \omega_n^{\frac{q_i}{p_i}}\epsilon^{-1}
$$
for~$i = 1, 2$.

Observing that~$\lambda = \lambda_1 + \lambda_2$, we have~$\beta_1 + \beta_2 = \lambda - \frac{\alpha}{p}-\frac{\epsilon}{q}$. Using elementary computations and the conditions~$q \lambda = q_1 \lambda_1 = q_2 \lambda_2$~and~$\frac{1}{q} =\frac{1}{q_1}+ \frac{1}{q_2}$, we have
$$
\left( \int_{\epsilon^{-1}}^\infty r^{(\beta_1 + \beta_2 +\frac{\alpha}{p} - \lambda)q - 1} dr \right)^{\frac{1}{q}} = \epsilon^{\frac{\epsilon-1}{q}},
$$
and
$$
\prod_{i=1}^2 \|f_{i}\|_{LML_{rad}^{\infty,\lambda_i,q_i}L_{ang}^{p_i}(\mathbb{R}^n, |x|^{\alpha})}= \omega_n^{\frac{1}{p_1}}\epsilon^{-\frac{1}{q_1}} \omega_n^{\frac{1}{p_2}}\epsilon^{-\frac{1}{q_2}}=\omega_n^{\frac{1}{p}}\epsilon^{-\frac{1}{q}} .
$$

Thus, for~$1 \leq q, q_i < \infty$, we obtain
\begin{align*}
&\|R_{\Phi}\|_{\prod\limits_{i=1}^2 LML_{rad}^{\infty,\lambda_i,q_i}L_{ang}^{p_i}(\mathbb{R}^n, |x|^{\alpha}) \to LML_{rad}^{\infty,\lambda,q}L_{ang}^{p}(\mathbb{R}^n, |x|^{\alpha})} \\
&\geq \frac{\|R_{\Phi}(f_{1}, f_{2})\|_{LML_{rad}^{\infty,\lambda,q}L_{ang}^{p}(\mathbb{R}^n, |x|^{\alpha})}}{\prod\limits_{i=1}^2 \|f_{i}\|_{LML_{rad}^{\infty,\lambda_i,q_i}L_{ang}^{p_i}(\mathbb{R}^n, |x|^{\alpha})}} \\
&\geq \epsilon^{\frac{\epsilon}{q}} \int_{|u_1| < \epsilon^{-1}} \int_{|u_2| < \epsilon^{-1}} \frac{\Phi(u_1, u_2)}{|u_1|^{n_1} |u_2|^{n_2}} \prod_{i=1}^2 |u_i|^{\frac{\epsilon}{q_i}+\frac{\alpha}{p_i} - \lambda_i} du_1 du_2.
\end{align*}

Letting~$\epsilon \to 0^+$, we get the desired result.
\end{proof}

We now establish a technical lemma, which plays a key role in the proof of Theorem \ref{th4}.

\begin{lemma}\label{th7}
Let ~$1 \leq p, p_i, \tilde{p}, \tilde{p_i} \leq \infty$, $0 < q, q_i \leq \infty$ and $\alpha\in\mathbb{R}$, let~$\lambda$~ satisfy (1.3) and $\lambda_i$~ satisfy
\begin{align*}
\lambda_i > 0 \text{ if } q_i < \infty \quad \text{and} \quad \lambda_i \geq 0 \text{ if } q_i = \infty
\tag{2.3}
\end{align*}
for ~$i = 1, \ldots, m$. Assume that ~$\Phi$~ is a nonnegative, locally integrable, radial function and
$$
g_{f_i}(y_i) = \omega_{n_i}^{-1} \int_{\mathbb{S}^{n_i-1}} f_i(|y_i| \xi_i) d\sigma(\xi_i), \quad y_i \in \mathbb{R}^{n_i},
$$
then all~$g_{f_i}$~are radial functions and
\begin{align*}
\frac{\|R_{\Phi}(\vec{f})\|_{LML_{rad}^{\tilde{p},\lambda,q}L_{ang}^{p}(\mathbb{R}^n, |x|^{\alpha})}}{\prod\limits_{i=1}^{m} \|f_i\|_{LML_{rad}^{\tilde{p_i},\lambda_i,q_i}L_{ang}^{p_i}(\mathbb{R}^n, |x|^{\alpha})}}
\leq \frac{\|R_{\Phi}(\vec{g_f})\|_{LML_{rad}^{\tilde{p},\lambda,q}L_{ang}^{p}(\mathbb{R}^n, |x|^{\alpha})}}{\prod\limits_{i=1}^{m} \|g_{f_i}\|_{LML_{rad}^{\tilde{p_i},\lambda_i,q_i}L_{ang}^{p_i}(\mathbb{R}^n, |x|^{\alpha})}},
\end{align*}
where and in what follows, $\omega_{n_i}$~is the surface area of the unit ball in~$\mathbb{R}^{n_i}$~for~$1 \leq i \leq m$,
$$
\vec{g_f} = (g_{f_1}, g_{f_2}, \ldots, g_{f_m}).
$$
\end{lemma}
\begin{proof}[Proof of Lemma \ref{th7}]
We restrict our attention to the case $m = 2$. It is clear that both $g_{f_1}$ and $g_{f_2}$ are radial functions. Since~$\Phi$~is radial, by using polar coordinates, similarly to \cite{An2023}, we can obtain
\begin{align*}
\widetilde{R_{\Phi}}(g_{f_1}, g_{f_2})(x)=\widetilde{R_{\Phi}}(f_1, f_2)(x).
\end{align*}

If~$0 < q_1 < \infty$, then we apply Minkowski's integral inequality and H\"{o}lder's inequality with~$1 \leq p_1,\tilde{p_1} < \infty$~to obtain
\begin{align*}
&\|g_{f_1}\|_{LML_{rad}^{\tilde{p_1},\lambda_1,q_1}L_{ang}^{p_1}(\mathbb{R}^n, |x|^{\alpha})} \\
&= \left( \int_0^\infty \left(\frac{1}{r^{\lambda_1}} \left( \int_0^r \left( \int_{\mathbb{S}^{n-1}} |g_{f_1}(\rho\theta)|^{p_1} d\sigma(\theta)\right)^{\frac{\tilde{p_1}}{p_1}} \rho^{n-1}\rho^\alpha d\rho \right)^{\frac{1}{\tilde{p_1}}} \right)^{q_1} \frac{dr}{r} \right)^{\frac{1}{q_1}} \\
&= \frac{1}{\omega_{n_1}}\Bigg( \int_0^\infty r^{-q_1 \lambda_1 - 1} \Bigg( \int_0^r \Bigg( \int_{\mathbb{S}^{n-1}} \Bigg| \int_{\mathbb{S}^{n_1-1}} f_1(|\rho\theta| \xi_1) d\sigma(\xi_1) \Bigg|^{p_1} d\sigma(\theta)\Bigg)^{\frac{\tilde{p_1}}{p_1}} \\
&\quad\times \rho^{n+\alpha-1} d\rho \Bigg)^{\frac{q_1}{\tilde{p_1}}} dr \Bigg)^{\frac{1}{q_1}} \\
&= \frac{\omega_{n}^{\frac{1}{p_1}}}{\omega_{n_1}} \left( \int_0^\infty r^{-q_1 \lambda_1 - 1} \left( \int_0^r \left| \int_{\mathbb{S}^{n_1-1}} f_1(\rho \xi_1) d\sigma(\xi_1) \right|^{\tilde{p_1}} \rho^{n+\alpha-1} d\rho \right)^{\frac{q_1}{\tilde{p_1}}} dr \right)^{\frac{1}{q_1}} \\
&\leq \frac{\omega_{n}^{\frac{1}{p_1}}}{\omega_{n_1}} \left( \int_0^\infty r^{-q_1 \lambda_1 - 1} \left( \int_0^r \left( \int_{\mathbb{S}^{n_1-1}} |f_1(\rho \xi_1)| d\sigma(\xi_1) \right)^{\tilde{p_1}} \rho^{n+\alpha-1} d\rho \right)^{\frac{q_1}{\tilde{p_1}}} dr \right)^{\frac{1}{q_1}}\\
&\leq \frac{\omega_{n}^{\frac{1}{p_1}}\omega_{n_1}^{\frac{1}{p_1'}}}{\omega_{n_1}} \left( \int_0^\infty r^{-q_1 \lambda_1 - 1} \left( \int_0^r \left( \int_{\mathbb{S}^{n_1-1}} |f_1(\rho \xi_1)|^{p_1} d\sigma(\xi_1) \right)^{\frac{\tilde{p_1}}{p_1}} \rho^{n+\alpha-1} d\rho \right)^{\frac{q_1}{\tilde{p_1}}} dr \right)^{\frac{1}{q_1}}\\
&\leq \frac{\omega_{n}^{\frac{1}{p_1}}\omega_{n_1}^{\frac{1}{p_1'}}}{\omega_{n_1}} \|f_1\|_{LML_{rad}^{\tilde{p_1},\lambda_1,q_1}L_{ang}^{p_1}(\mathbb{R}^n, |x|^{\alpha})}.
\end{align*}

With this, we have
$$
\|g_{f_1}\|_{LML_{rad}^{\tilde{p_1},\lambda_1,q_1}L_{ang}^{p_1}(\mathbb{R}^n, |x|^{\alpha})} \leq \frac{\omega_{n}^{\frac{1}{p_1}}\omega_{n_1}^{\frac{1}{p_1'}}}{\omega_{n_1}} \|f_1\|_{LML_{rad}^{\tilde{p_1},\lambda_1,q_1}L_{ang}^{p_1}(\mathbb{R}^n, |x|^{\alpha})}.
$$

If~$\tilde{p_1} = \infty$~and~$1 \leq p_1 < \infty$, there holds
\begin{align*}
&\|g_{f_1}\|_{LML_{rad}^{\infty,\lambda_1,q_1}L_{ang}^{p_1}(\mathbb{R}^n, |x|^{\alpha})} \\
&= \left( \int_0^\infty r^{-q_1 \lambda_1 - 1} \left( \operatorname*{ess\,sup}_{0<\rho<r}\Bigg( \int_{\mathbb{S}^{n-1}} |g_{f_1}(\rho\theta)|^{p_1}\rho^\alpha d\sigma(\theta) \right)^{\frac{1}{p_1}}\Bigg)^{q_1} dr \right)^{\frac{1}{q_1}} \\
&= \frac{1}{\omega_{n_1}} \left( \int_0^\infty r^{-q_1 \lambda_1 - 1} \left( \operatorname*{ess\,sup}_{0<\rho<r}\Bigg( \int_{\mathbb{S}^{n-1}} \left| \int_{\mathbb{S}^{n_1-1}} f_1(|\rho\theta| \xi_1) d\sigma(\xi_1) \right|^{p_1} \rho^\alpha d\sigma(\theta) \right)^{\frac{1}{p_1}} \Bigg)^{q_1} dr \right)^{\frac{1}{q_1}} \\
&\leq \frac{\omega_n^{\frac{1}{p_1}}}{\omega_{n_1}} \left( \int_0^\infty r^{-q_1 \lambda_1 - 1} \left( \operatorname*{ess\,sup}_{0<\rho<r}\Bigg( \Bigg(\int_{\mathbb{S}^{n_1-1}} \left| f_1(\rho \xi_1)\right| d\sigma(\xi_1)\Bigg)^{p_1} \rho^\alpha  \right)^{\frac{1}{p_1}} \Bigg)^{q_1} dr \right)^{\frac{1}{q_1}} \\
&\leq \frac{\omega_n^{\frac{1}{p_1}}\omega_{n_1}^{\frac{1}{p_1'}}}{\omega_{n_1}} \left( \int_0^\infty r^{-q_1 \lambda_1 - 1} \left( \operatorname*{ess\,sup}_{0<\rho<r}\Bigg( \int_{\mathbb{S}^{n_1-1}} \left| f_1(\rho \xi_1)\right|^{p_1} d\sigma(\xi_1) \rho^\alpha  \right)^{\frac{1}{p_1}} \Bigg)^{q_1} dr \right)^{\frac{1}{q_1}} \\
&= \frac{\omega_n^{\frac{1}{p_1}}\omega_{n_1}^{\frac{1}{p_1'}}}{\omega_{n_1}}\|f_1\|_{LML_{rad}^{\infty,\lambda_1,q_1}L_{ang}^{p_1}(\mathbb{R}^n, |x|^{\alpha})}.
\end{align*}

Employing a similar approach, we deduce that
$$
\|g_{f_1}\|_{LML_{rad}^{\tilde{p_1},\lambda_1,\infty}L_{ang}^{p_1}(\mathbb{R}^n, |x|^{\alpha})} \leq \|f_1\|_{LML_{rad}^{\tilde{p_1},\lambda_1,\infty}L_{ang}^{p_1}(\mathbb{R}^n, |x|^{\alpha})}
$$
for~$1 \leq \tilde{p_1}\leq \infty$~and~$1 \leq p_1< \infty$, and also
$$
\|g_{f_2}\|_{LML_{rad}^{\tilde{p_2},\lambda_2,q_2}L_{ang}^{p_2}(\mathbb{R}^n, |x|^{\alpha})} \leq \|f_2\|_{LML_{rad}^{\tilde{p_2},\lambda_2,q_2}L_{ang}^{p_2}(\mathbb{R}^n, |x|^{\alpha})}
$$
for~$1 \leq \tilde{p_2}\leq \infty$, $1 \leq p_2 < \infty$~and~$0 < q_2 \leq \infty$. Therefore, one has that
$$
\frac{\|R_{\Phi}(f_1, f_2)\|_{LML_{rad}^{\tilde{p},\lambda,q}L_{ang}^{p}(\mathbb{R}^n, |x|^{\alpha})}}{\prod\limits_{i=1}^2 \|f_i\|_{LML_{rad}^{\tilde{p_i},\lambda_i,q_i}L_{ang}^{p_i}(\mathbb{R}^n, |x|^{\alpha})}} \leq \frac{\|R_{\Phi}(g_{f_1}, g_{f_2})\|_{LML_{rad}^{\tilde{p},\lambda,q}L_{ang}^{p}(\mathbb{R}^n, |x|^{\alpha})}}{\prod\limits_{i=1}^2 \|g_{f_i}\|_{LML_{rad}^{\tilde{p_i},\lambda_i,q_i}L_{ang}^{p_i}(\mathbb{R}^n, |x|^{\alpha})}}.
$$
The proof of Lemma \ref{th7} is completed.
\end{proof}
Lemma \ref{th7} establishes the equality of the operator norms of~$R_{\Phi}$~and its restriction to radial functions on the space~$LML_{rad}^{\tilde{p},\lambda,q}L_{ang}^{p}(\mathbb{R}^n, |x|^{\alpha})$. Consequently, we may assume without loss of generality that the functions~$f_i$~for~$i = 1, 2$~are radial in the subsequent proof.

\begin{proof}[Proof of Theorem \ref{th4}]
By Lemma \ref{th7}, we assume that $f_i$ are radial functions for $i = 1, 2$. By expressing the function $\widetilde{R_{\Phi}}(f_1, f_2)$ in polar coordinates, similar to \cite{An2023}, we can obtain
\begin{align*}
\widetilde{R_{\Phi}}(f_1, f_2)(x) = \omega_{n_1} \omega_{n_2} \int_0^\infty \int_0^\infty \prod_{i=1}^2 f_i\left( \frac{|x|}{t_i} \right) \frac{\Phi(t_1, t_2)}{t_1 t_2} dt_1 dt_2,
\end{align*}
where we have used that \(\Phi\) is a radial function. For \(1 \leq p, p_i < \infty\), by Minkowski's integral inequality and H\"{o}lder's inequality, we derive
\begin{align*}
&\left( \int_0^r \left( \int_{\mathbb{S}^{n-1}} |\widetilde{R_{\Phi}}(f_1, f_2)(\rho\theta)|^p d\sigma(\theta) \right)^{\frac{\tilde{p}}{p}}\rho^{n+\alpha-1}  d\rho \right)^{\frac{1}{\tilde{p}}} \\
&=\omega_{n_1} \omega_{n_2}\left( \int_0^r \left( \int_{\mathbb{S}^{n-1}} \left|\int_0^\infty \int_0^\infty \prod_{i=1}^2 f_i\left( \frac{|\rho\theta|}{t_i} \right) \frac{\Phi(t_1, t_2)}{t_1 t_2} dt_1 dt_2 \right|^p d\sigma(\theta) \right)^{\frac{\tilde{p}}{p}}\rho^{n+\alpha-1}  d\rho \right)^{\frac{1}{\tilde{p}}}\\
& \leq \omega_{n_1} \omega_{n_2} \int_0^\infty \int_0^\infty \prod_{i=1}^2 \Bigg( \int_0^r \Bigg( \int_{\mathbb{S}^{n-1}}\left|f_i\left( \frac{|\rho\theta|}{t_i} \right)\right|^p d\sigma(\theta)\Bigg)^{\frac{\tilde{p}}{p}} \rho^{n+\alpha-1}  d\rho\Bigg)^{\frac{1}{\tilde{p}}} \frac{\Phi(t_1, t_2)}{t_1 t_2} dt_1 dt_2 \\
& \leq \omega_{n_1} \omega_{n_2} \int_0^\infty \int_0^\infty \prod_{i=1}^2 \Bigg( \int_0^{\frac{r}{t_i}} \Bigg( \int_{\mathbb{S}^{n-1}}\left|f_i\left(y_i|\theta|\right)\right|^{p_i} d\sigma(\theta)\Bigg)^{\frac{\tilde{p_i}}{p_i}} {y_i}^{n+\alpha-1}  dy_i\Bigg)^{\frac{1}{\tilde{p_i}}} \frac{\Phi(t_1, t_2)}{t_1 t_2}\\
&\quad\times{t_1}^{\frac{n+\alpha}{\tilde{p_1}}}{t_2}^{\frac{n+\alpha}{\tilde{p_2}}}  dt_1 dt_2.
\end{align*}

The rest of the proof runs as before, so we omit the repetitive details, and the proof of Theorem \ref{th7} is finished.
\end{proof}

\begin{proof}[Proof of Theorem \ref{th5}]
The proof of this theorem follows by similar arguments to the ones in the proof of Theorem \ref{th3}.
\end{proof}

\begin{lemma}\label{th8}
Let~$1 \leq p, p_i, \tilde{p}, \tilde{p_i} \leq \infty$, $0 < q, q_i \leq \infty$ and $\alpha\in\mathbb{R}$, let~$\lambda$~satisfy (1.3) and $\lambda_i$~satisfy (2.3) for~$i = 1, \ldots, m$. Suppose that~$\Phi$~be a nonnegative locally integrable function and
$$
g_{f_i}(y_i) = \omega_{n_i}^{-1} \int_{|\xi_i|=1} f_i(|y_i| \xi_i) d\xi_i, \quad y_i \in \mathbb{R}^{n_i}.
$$

Then all~$g_{f_i}$~are radial functions and
$$
\frac{\|\widetilde{S_{\Phi}}(\vec{f})\|_{LML_{rad}^{\tilde{p},\lambda,q}L_{ang}^{p}(\mathbb{R}^n, |x|^{\alpha})}}{\prod\limits_{i=1}^{m} \|f_i\|_{LML_{rad}^{\tilde{p_i},\lambda_i,q_i}L_{ang}^{p_i}(\mathbb{R}^n, |x|^{\alpha})}} \leq \frac{\|\widetilde{S_{\Phi}}(\vec{g_f})\|_{LML_{rad}^{\tilde{p},\lambda,q}L_{ang}^{p}(\mathbb{R}^n, |x|^{\alpha})}}{\prod\limits_{i=1}^{m} \|g_{f_i}\|_{LML_{rad}^{\tilde{p_i},\lambda_i,q_i}L_{ang}^{p_i}(\mathbb{R}^n, |x|^{\alpha})}}.
$$
\end{lemma}

The proof of this lemma is similar to the proof of Lemma \ref{th7}.

\begin{proof}[Proof of Theorem \ref{th6}]
By Lemma \ref{th8}, we assume that~$f_1$~and~$f_2$~are radial functions. Similar to \cite{An2023}, let us write~$\widetilde{S_{\Phi}}(f_1, f_2)$~in the form
\begin{align*}
&\widetilde{S_{\Phi}}(f_1, f_2)(x)\\
&= \omega_{n_1} \omega_{n_2} \int_0^\infty \int_0^\infty \Phi\left( \frac{t_1 t_2}{\sqrt{t_1^2 + t_2^2}} \right)\frac{{t_1}^{n_1}{t_2}^{n_2}}{({t_1}^2+{t_2}^2)^{\frac{n_1+n_2}{2}}}  f_1\left( \frac{|x|}{t_1} \right) f_2\left( \frac{|x|}{t_2} \right) \frac{dt_1}{t_1} \frac{dt_2}{t_2}.
\end{align*}

For~$1 \leq p, p_i,\tilde{p}, \tilde{p_i} < \infty$, by Minkowski's integral inequality and H\"{o}lder's inequality, we have
\begin{align*}
&\left( \int_0^r \left( \int_{\mathbb{S}^{n-1}} |\widetilde{S_{\Phi}}(f_1, f_2)(x)|^p d\sigma(\theta) \right)^{\frac{\tilde{p}}{p}}\rho^{n+\alpha-1}  d\rho \right)^{\frac{1}{p}} \\
&=\omega_{n_1} \omega_{n_2}\Bigg( \int_0^r \Bigg( \int_{\mathbb{S}^{n-1}} \Bigg|\int_0^\infty \int_0^\infty \Phi\Bigg( \frac{t_1 t_2}{\sqrt{t_1^2 + t_2^2}} \Bigg)\frac{{t_1}^{n_1}{t_2}^{n_2}}{({t_1}^2+{t_2}^2)^{\frac{n_1+n_2}{2}}}  f_1\Bigg( \frac{|\rho\theta|}{t_1} \Bigg)\\
&\quad \times f_2\Bigg( \frac{|\rho\theta|}{t_2} \Bigg) \frac{dt_1}{t_1} \frac{dt_2}{t_2}\Bigg|^p d\sigma(\theta)\Bigg)^{\frac{\tilde{p}}{p}}\rho^{n+\alpha-1}  d\rho \Bigg)^{\frac{1}{\tilde{p}}} \\
&\leq\omega_{n_1} \omega_{n_2}\int_0^\infty \int_0^\infty \Phi\Bigg( \frac{t_1 t_2}{\sqrt{t_1^2 + t_2^2}} \Bigg)\frac{{t_1}^{n_1}{t_2}^{n_2}}{({t_1}^2+{t_2}^2)^{\frac{n_1+n_2}{2}}} \prod_{i=1}^2 \Bigg( \int_0^r\Bigg( \int_{\mathbb{S}^{n-1}} \Bigg|f_i\Bigg( \frac{|\rho\theta|}{t_i} \Bigg)\Bigg|^{p}\\
&\quad \times d\sigma(\theta)\Bigg)^{\frac{\tilde{p}}{p}}\rho^{n+\alpha-1}  d\rho \Bigg)^{\frac{1}{\tilde{p}}} \frac{dt_1}{t_1}\frac{dt_2}{t_2}\\
&=\omega_{n_1} \omega_{n_2}\int_0^\infty \int_0^\infty \Phi\Bigg( \frac{t_1 t_2}{\sqrt{t_1^2 + t_2^2}} \Bigg)\frac{{t_1}^{n_1}{t_2}^{n_2}}{({t_1}^2+{t_2}^2)^{\frac{n_1+n_2}{2}}} \prod_{i=1}^2 \Bigg( \int_0^{\frac{r}{t_i}}\Bigg( \int_{\mathbb{S}^{n-1}} \Bigg|f_i\Bigg( y_i|\theta|\Bigg)\Bigg|^{p_i}\\
&\quad \times d\sigma(\theta) \Bigg)^{\frac{\tilde{p_i}}{p_i}} {y_i}^{n+\alpha-1}  dy_i \Bigg)^{\frac{1}{\tilde{p_i}}} {t_1}^{\frac{n+\alpha}{\tilde{p_1}}-1} {t_2}^{\frac{n+\alpha}{\tilde{p_2}}-1}  dt_1dt_2.
\end{align*}

The remaining proof is just like before, so the specific details can be omitted, and the proof of Theorem \ref{th6} is finished.
\end{proof}

\section{Applications}

As previously noted, a primary reason for investigating Hausdorff operators stems from the fact that numerous classical operators in analysis can be recovered as particular instances of the general Hausdorff operator through appropriate selections of the function~$\Phi$. These classical operators include the \(n\)-dimensional Hardy operator and its dual operator, the multilinear weighted Hardy operator, and the multilinear Ces\`{a}ro operator.

If we take
$$
\Phi(x) = \frac{\chi_{\{|x| > 1\}}(x)}{\nu_{n}|x|^{n}} \quad \text{and} \quad \Phi(x) = \frac{1}{\nu_{n}}\chi_{\{|x| < 1\}}(x)
$$
in the $\widetilde{\mathcal{H}_{\Phi}}$, then we get the $n$-dimensional Hardy operator
$$
\mathcal{H}(f)(x) = \frac{1}{|B(0, |x|)|} \int_{B(0, |x|)} f(y) dy, \quad x \in \mathbb{R}^{n} \setminus \{0\},
$$
and the dual operator
$$
\mathcal{H}^{*}(f)(x) = \int_{|y| > |x|} \frac{f(y)}{|B(0, |y|)|} dy, \quad x \in \mathbb{R}^{n}.
$$
Here and below, $\nu_n$ is the volume of the unit ball in $\mathbb{R}^n$.
\begin{theorem}
Let~$1 \leq p, \tilde{p}, q \leq \infty$, $\alpha\in\mathbb{R}$, $n \tilde{p} + \lambda \tilde{p} - (n+\alpha)> 0$~and (1.3) be satisfied. Then the $n$-dimensional Hardy operator~$\mathcal{H}$~is bounded on mixed radial-angular local Morrey-type space $LML_{rad}^{\tilde{p},\lambda,q}L_{ang}^{p}(\mathbb{R}^n, |x|^{\alpha})$ with the power weight, that is,
$$
\|\mathcal{H}(f)\|_{LML_{rad}^{\tilde{p},\lambda,q}L_{ang}^{p}(\mathbb{R}^n, |x|^{\alpha})} \leq \frac{n \tilde{p}}{n \tilde{p} + \lambda \tilde{p} - (n+\alpha)} \|f\|_{LML_{rad}^{\tilde{p},\lambda,q}L_{ang}^{p}(\mathbb{R}^n, |x|^{\alpha})}
$$
for~$1 \leq p < \infty$, and
$$
\|\mathcal{H}(f)\|_{LML_{rad}^{\infty,\lambda,q}L_{ang}^{p}(\mathbb{R}^n, |x|^{\alpha})} \leq \frac{n}{n + \lambda} \|f\|_{LML_{rad}^{\infty,\lambda,q}L_{ang}^{p}(\mathbb{R}^n, |x|^{\alpha})}.
$$
Moreover,
$$
\|\mathcal{H}\|_{LML_{rad}^{\tilde{p},\lambda,q}L_{ang}^{p}(\mathbb{R}^n, |x|^{\alpha}) \to LML_{rad}^{\tilde{p},\lambda,q}L_{ang}^{p}(\mathbb{R}^n, |x|^{\alpha})}=\frac{n \tilde{p}}{n \tilde{p} + \lambda \tilde{p} - (n+\alpha)}
$$
for~$1 \leq p < \infty$, and
$$
\|\mathcal{H}\|_{LML_{rad}^{\infty,\lambda,q}L_{ang}^{p}(\mathbb{R}^n, |x|^{\alpha}) \to LML_{rad}^{\infty,\lambda,q}L_{ang}^{p}(\mathbb{R}^n, |x|^{\alpha})} = \frac{n}{n + \lambda}.
$$
\end{theorem}

\begin{theorem}
Let~$1 \leq p, \tilde{p} < \infty$, $1 \leq q \leq \infty$, $\alpha\in\mathbb{R}$, $0 \leq \lambda < \frac{n+\alpha}{\tilde{p}}$, and (1.3) be satisfied. Then the dual Hardy operator~$\mathcal{H}^*$~is bounded on mixed radial-angular local Morrey-type space $LML_{rad}^{\tilde{p},\lambda,q}L_{ang}^{p}(\mathbb{R}^n, |x|^{\alpha})$ with the power weight, that is,
$$
\|\mathcal{H}^*(f)\|_{LML_{rad}^{\tilde{p},\lambda,q}L_{ang}^{p}(\mathbb{R}^n, |x|^{\alpha})} \leq \frac{n \tilde{p}}{n+\alpha - \lambda \tilde{p}} \|f\|_{LML_{rad}^{\tilde{p},\lambda,q}L_{ang}^{p}(\mathbb{R}^n, |x|^{\alpha})}.
$$

Moreover,
$$
\|\mathcal{H}^*\|_{LML_{rad}^{\tilde{p},\lambda,q}L_{ang}^{p}(\mathbb{R}^n, |x|^{\alpha}) \to LML_{rad}^{\tilde{p},\lambda,q}L_{ang}^{p}(\mathbb{R}^n, |x|^{\alpha})} = \frac{n \tilde{p}}{n +\alpha- \lambda \tilde{p}}.
$$
\end{theorem}

Now, let's make $\phi : [0, 1] \rightarrow [0, \infty)$ a measurable function. If we take
$$
\Phi(y) = (\omega_n |y|)^{-1} \phi(|y|^{-1}) \chi_{\{|y| > 1\}}(y)
$$
in the $\mathcal{H}_{\Phi}$, then we get the weighted Hardy-Littlewood average operator \cite{carton1984} (see also \cite{fu2015, xiao2001})
$$
H_{\phi}(f)(x) = \int_0^1 \phi(t) f(t x) dt, \quad x \in \mathbb{R}^n.
$$

If we take
$$
\Phi(y) = \omega_n^{-1} |y|^{-n+1} \phi(|y|) \chi_{\{0 < |y| < 1\}}(y)
$$
in $\mathcal{H}_{\Phi}$, we obtain the weighted Ces\`{a}ro operator
$$
G_{\phi}(f)(x) = \int_0^1 f\left( \frac{x}{t} \right) t^{-n} \phi(t) dt, \quad x \in \mathbb{R}^n.
$$


Let
$$
\phi : \overbrace{[0, 1] \times [0, 1] \times \cdots \times [0, 1]}^{m\; time}\rightarrow [0, \infty)
$$
be a measurable function. If we take
$$
\Phi(\vec{y}) = \phi(|y_1|^{-1}, |y_2|^{-1}, \ldots, |y_m|^{-1}) \prod_{i=1}^m (\omega_{n_i} |y_i|)^{-1} \chi_{\{|y_i| > 1\}}(y_i)
$$
in~$R_{\Phi}$, we get the multilinear weighted Hardy operator (see \cite{fu2015})
$$
H^m_{\phi}(\vec{f})(x) = \int_0^1 \cdots \int_0^1 \phi(\vec{t}) \prod_{i=1}^m f_i(t_i x) d\vec{t}, \quad x \in \mathbb{R}^n.
$$

If we take
$$
\Phi(\vec{y}) = \phi(|y_1|, |y_2|, \ldots, |y_m|) \prod_{i=1}^m (\omega_{n_i}^{-1} |y_i|^{-n_i+1}) \chi_{\{|y_i| < 1\}}(y_i)
$$
in~$R_{\Phi}$, we obtain the weighted multilinear Ces\`{a}ro operator (see \cite{fu2015})
$$
G^m_{\phi}(\vec{f})(x) = \int_{0 < t_1, t_2, \ldots, t_m < 1} \left( \prod_{i=1}^m f_i \left( \frac{x}{t_i} \right) t_i^{-n_i} \right) \phi(\vec{t}) d\vec{t}, \quad x \in \mathbb{R}^n.
$$

As applications of  Theorems \ref{th2} and \ref{th3}, we get the following results.

\begin{theorem}
(i) Let~$1 \leq \tilde{p},p, q \leq \infty$, $\alpha\in\mathbb{R}$ and (1.3) be satisfied. Then the weighted Hardy operator~$H_{\phi}$~is bounded on mixed radial-angular local Morrey-type space $LML_{rad}^{\tilde{p},\lambda,q}L_{ang}^{p}(\mathbb{R}^n, |x|^{\alpha})$ with the power weight, that is,
$$
\|H_{\phi}(f)\|_{LML_{rad}^{\tilde{p},\lambda,q}L_{ang}^{p}(\mathbb{R}^n, |x|^{\alpha})} \leq \int_0^1 \phi(t) t^{\lambda - \frac{n+\alpha}{\tilde{p}}} dt \|f\|_{LML_{rad}^{\tilde{p},\lambda,q}L_{ang}^{p}(\mathbb{R}^n, |x|^{\alpha})}
$$
if and only if
$$
\int_0^1 \phi(t) t^{\lambda - \frac{n+\alpha}{\tilde{p}}} dt < \infty.
$$
Moreover,
$$
\|H_{\phi}\|_{LML_{rad}^{\tilde{p},\lambda,q}L_{ang}^{p}(\mathbb{R}^n, |x|^{\alpha}) \to LML_{rad}^{\tilde{p},\lambda,q}L_{ang}^{p}(\mathbb{R}^n, |x|^{\alpha})} = \int_0^1 \phi(t) t^{\lambda - \frac{n+\alpha}{\tilde{p}}} dt.
$$
(ii) Let~$1 \leq p,  p_i, \tilde{p}, \tilde{p_i} < \infty$, $1 \leq q \leq \infty$ satisfy $\frac{1}{p} = \sum\limits_{i=1}^{m} \frac{1}{p_i}$ and $\frac{1}{\tilde{p}} = \sum\limits_{i=1}^{m} \frac{1}{\tilde{p_i}}$. Assume that $\alpha\in\mathbb{R}$, $\lambda = \sum\limits_{i=1}^m \lambda_i$~for~$i = 1, \ldots, m$, $\lambda$ satisfies (1.3), and $\lambda_i$ satisfies (2.1). If
$$
\int_0^1 \cdots \int_0^1 \phi(\vec{t}) \prod_{i=1}^m t_i^{\lambda_i - \frac{n+\alpha}{\tilde{p_i}}} d\vec{t} < \infty,
$$
then
\begin{align*}
&\|H^m_{\phi}(\vec{f})\|_{LML_{rad}^{\tilde{p},\lambda,q}L_{ang}^{p}(\mathbb{R}^n, |x|^{\alpha})} \\
&\leq \int_0^1 \cdots \int_0^1 \phi(\vec{t}) \prod_{i=1}^m t_i^{\lambda_i - \frac{n+\alpha}{\tilde{p_i}}} d\vec{t} \prod_{i=1}^m \|f_i\|_{LML_{rad}^{\tilde{p_i},\lambda_i,(q\tilde{p_i})/\tilde{p}}L_{ang}^{p_i}(\mathbb{R}^n, |x|^{\alpha})}.
\end{align*}
Moreover, if
$$
p_i \lambda_i = p_j \lambda_j, \quad i, j = 1, \ldots, m,
$$
then
$$
\|H^m_{\phi}\|_{\prod\limits_{i=1}^m LML_{rad}^{\tilde{p_i},\lambda_i,(q\tilde{p_i})/\tilde{p}}L_{ang}^{p_i}(\mathbb{R}^n, |x|^{\alpha}) \to LML_{rad}^{\tilde{p},\lambda,q}L_{ang}^{p}(\mathbb{R}^n, |x|^{\alpha})} = \int_0^1 \cdots \int_0^1 \phi(\vec{t}) \prod_{i=1}^m t_i^{\lambda_i - \frac{n+\alpha}{\tilde{p_i}}} d\vec{t}.
$$
\end{theorem}

\begin{theorem}
(i) Let~$1 \leq \tilde{p},p, q \leq \infty$, $\alpha\in\mathbb{R}$ and (1.3) be satisfied. Then the weighted Ces\`{a}ro operator~$G_{\phi}$ is bounded on mixed radial-angular local Morrey-type space $LML_{rad}^{\tilde{p},\lambda,q}L_{ang}^{p}(\mathbb{R}^n, |x|^{\alpha})$ with the power weight, that is,
$$
\|G_{\phi}(f)\|_{LML_{rad}^{\tilde{p},\lambda,q}L_{ang}^{p}(\mathbb{R}^n, |x|^{\alpha})} \leq \int_0^1 \phi(t) t^{\frac{n+\alpha}{\tilde{p}} - \lambda - n} dt \|f\|_{LML_{rad}^{\tilde{p},\lambda,q}L_{ang}^{p}(\mathbb{R}^n, |x|^{\alpha})}
$$
if and only if
$$
\int_0^1 \phi(t) t^{\frac{n+\alpha}{\tilde{p}} - \lambda - n} dt < \infty.
$$
Moreover,
$$
\|G_{\phi}\|_{LML_{rad}^{\tilde{p},\lambda,q}L_{ang}^{p}(\mathbb{R}^n, |x|^{\alpha}) \to LML_{rad}^{\tilde{p},\lambda,q}L_{ang}^{p}(\mathbb{R}^n, |x|^{\alpha})} = \int_0^1 t^{\frac{n+\alpha}{\tilde{p}} - \lambda - n} \phi(t) dt.
$$
(ii) Let~$1 \leq p,  p_i, \tilde{p}, \tilde{p_i} < \infty$, $1 \leq q \leq \infty$ satisfy $\frac{1}{p} = \sum\limits_{i=1}^{m} \frac{1}{p_i}$ and $\frac{1}{\tilde{p}} = \sum\limits_{i=1}^{m} \frac{1}{\tilde{p_i}}$. Assume that $\alpha\in\mathbb{R}$, $\lambda = \sum\limits_{i=1}^m \lambda_i$~for~$i = 1, \ldots, m$, $\lambda$ satisfies (1.3), and $\lambda_i$ satisfies (2.1). If
$$
\int_0^1 \cdots \int_0^1 \phi(\vec{t}) \prod_{i=1}^m t_i^{\frac{n+\alpha}{\tilde{p_i}} - \lambda_i - n_i} d\vec{t} < \infty,
$$
then
\begin{align*}
&\|G^m_{\phi}(\vec{f})\|_{LML_{rad}^{\tilde{p},\lambda,q}L_{ang}^{p}(\mathbb{R}^n, |x|^{\alpha})} \\
&\leq \int_0^1 \cdots \int_0^1 \phi(\vec{t}) \prod_{i=1}^m t_i^{\frac{n+\alpha}{\tilde{p_i}} - \lambda_i - n_i} d\vec{t} \prod_{i=1}^m \|f_i\|_{LML_{rad}^{\tilde{p_i},\lambda_i,(q\tilde{p_i})/\tilde{p}}L_{ang}^{p_i}(\mathbb{R}^n, |x|^{\alpha})}.
\end{align*}
Moreover, if
$$
p_i \lambda_i = p_j \lambda_j, \quad i, j = 1, \ldots, m,
$$
then
\begin{align*}
&\|G^m_{\phi}\|_{\prod\limits_{i=1}^m LML_{rad}^{\tilde{p_i},\lambda_i,(q\tilde{p_i})/\tilde{p}}L_{ang}^{p_i}(\mathbb{R}^n, |x|^{\alpha}) \to LML_{rad}^{\tilde{p},\lambda,q}L_{ang}^{p}(\mathbb{R}^n, |x|^{\alpha})} \\
&= \int_0^1 \cdots \int_0^1 \phi(\vec{t}) \prod_{i=1}^m t_i^{\frac{n+\alpha}{\tilde{p_i}} - \lambda_i - n_i} d\vec{t}.
\end{align*}
\end{theorem}

In addition, Liu et al. \cite{LZ2025} establish sharp bounds for the $m$-linear Hardy operator on the mixed radial-angular homogeneous central total Morrey spaces. Naturally, we will consider to establishing sharp bounds for the Hausdorff operator $\widetilde{\mathcal{H}_{\Phi}}(f)$ on the mixed radial-angular homogeneous central Morrey space and the mixed radial-angular homogeneous central total Morrey space.

\begin{definition}\cite{LZ2025}
Let $1 < \tilde{p}, p < \infty$, $0\leq\lambda<n$ and $f \in L^{\tilde{p}}_{\text{rad}, \text{loc}} L^{p}_{\text{ang}}((0, \infty) \times S^{n-1})$. A function $f$ is said to belong to the mixed radial-angular homogeneous central Morrey space $\mathcal{\dot{B}}{L}^{\tilde{p}, \lambda}_{\text{rad}} {L}^{p}_{\text{ang}}(\mathbb{R}^n)$, if
$$
\| f \|_{\mathcal{\dot{B}}{L}^{\tilde{p}, \lambda}_{\text{rad}} {L}^{p}_{\text{ang}}(\mathbb{R}^n)} = \sup_{t > 0} t^{-\frac{\lambda}{\tilde{p}}}\left(\int_0^t \left( \int_{S^{n-1}} |f(\rho \theta)|^{p} d\sigma(\theta) \right)^{\frac{\tilde{p}}{p}} \rho^{n - 1} d\rho \right)^{\frac{1}{\tilde{p}}} < \infty.
$$
\end{definition}

\begin{definition}\cite{LZ2025}
Let $1 < \tilde{p}, p < \infty$, $0\leq\lambda<n$ and $f \in L^{\tilde{p}}_{\text{rad}, \text{loc}} L^{p}_{\text{ang}}((0, \infty) \times S^{n-1})$. A function $f$ is said to belong to the mixed radial-angular homogeneous central total Morrey space ${\mathcal{\dot{B}}}{L}^{\tilde{p}, \lambda, \mu}_{\text{rad}} {L}^{p}_{\text{ang}}(\mathbb{R}^n)$, if
$$
\| f \|_{{\mathcal{\dot{B}}}{L}^{\tilde{p}, \lambda, \mu}_{\text{rad}} {L}^{p}_{\text{ang}}(\mathbb{R}^n)} = \sup_{t > 0}[t]_1^{-\frac{\lambda}{\tilde{p}}} \left[\frac{1}{t}\right]_1^{\frac{\mu}{\tilde{p}}} \left( \int_0^t \left( \int_{S^{n-1}} |f(\rho \theta)|^{p} d\sigma(\theta) \right)^{\frac{\tilde{p}}{p}} \rho^{n - 1} d\rho \right)^{\frac{1}{\tilde{p}}} < \infty.
$$
\end{definition}

\begin{theorem}\label{th16}
Let~$f \in {\mathcal{\dot{B}}}{L}_{rad}^{\tilde{p},\lambda}{L}_{ang}^{p_1}(\mathbb{R}^n)$, $1 < p, p_1, \tilde{p}< \infty$, and $0<\lambda<n$. Then the operator $\widetilde{\mathcal{H}_{\Phi}}(f)$ maps ${\mathcal{\dot{B}}}{L}_{rad}^{\tilde{p},\lambda}{L}_{ang}^{p_1}(\mathbb{R}^n)$ to ${\mathcal{\dot{B}}}{L}_{rad}^{\tilde{p},\lambda}{L}_{ang}^{p}(\mathbb{R}^n)$ with norm equal to the constant
$$
\omega_n^{\frac{1}{p}-\frac{1}{p_1}}\frac{\tilde{p}n}{\lambda-n+\tilde{p}n}.
$$
\end{theorem}

\begin{theorem}\label{th17}
Let $f \in {\mathcal{\dot{B}}}{L}_{rad}^{\tilde{p},\lambda,\mu}{L}_{ang}^{p_1}(\mathbb{R}^n)$, $0 \leq \mu<\lambda < n$, $1 < p, p_1, \tilde{p} < \infty$. Then the operator $\widetilde{\mathcal{H}_{\Phi}}(f)$ maps ${\mathcal{\dot{B}}}{L}_{rad}^{\tilde{p},\lambda,\mu}{L}_{ang}^{p_1}(\mathbb{R}^n)$ to $\mathcal{\dot{B}} {L}_{rad}^{\tilde{p},\lambda,\mu}{L}_{ang}^{p}(\mathbb{R}^n)$ with norm equal to the constant
$$
\omega_n^{\frac{1}{p}-\frac{1}{p_1}}\frac{\tilde{p}n}{\mu-n+\tilde{p}n}.
$$
\end{theorem}

\begin{theorem}\label{th18}
Let $f \in {\mathcal{\dot{B}}}{L}_{rad}^{\tilde{p},\lambda,\mu}{L}_{ang}^{p_1}(\mathbb{R}^n)$, $0 \leq \lambda<\mu < n$, $1 < p, p_1, \tilde{p} < \infty$. Then the operator $\widetilde{\mathcal{H}_{\Phi}}(f)$ maps ${\mathcal{\dot{B}}}{L}_{rad}^{\tilde{p},\lambda,\mu}{L}_{ang}^{p_1}(\mathbb{R}^n)$ to $\mathcal{\dot{B}} {L}_{rad}^{\tilde{p},\lambda,\mu}{L}_{ang}^{p}(\mathbb{R}^n)$ with norm equal to the constant
$$
\omega_n^{\frac{1}{p}-\frac{1}{p_1}}\frac{\tilde{p}n}{\lambda-n+\tilde{p}n}.
$$
\end{theorem}

By combining the previous ideas obtained the results of on mixed radial-angular homogeneous central total Morrey space and the method in \cite{LZ2025}, the proof of Theorems \ref{th16}, \ref{th17} and \ref{th18} are only simple imitation.

\section{Supplementary results}
To investigate the complementary local Morrey-type space, we will modify the above four multilinear Hausdorff operators \cite{Wei2025}.

The first multilinear Hausdorff operator is modified to
$$
\dot{R}_{\Phi}(\vec{f})(x) = \int_{\mathbb{R}^{nm}} \frac{\Phi(\vec{u})}{\prod\limits_{i=1}^m |u_i|^n} \prod_{i=1}^m f_i\left(\frac{x}{|u_i|}\right) d\vec{u}.
$$

The second multilinear Hausdorff operator is modified to
$$
\widetilde{\dot{R}_{\Phi}}(\vec{f})(x) = \int_{\mathbb{R}^{nm}} \frac{\Phi\left(\frac{x}{|u_1|}, \frac{x}{|u_2|}, \cdots, \frac{x}{|u_m|}\right)}{\prod\limits_{i=1}^m |u_i|^n} \prod_{i=1}^m f_i(u_i) d\vec{u}.
$$

The third multilinear Hausdorff operator is modified to
$$
\dot{S}_{\Phi}(\vec{f})(x) = \int_{\mathbb{R}^{nm}} \frac{\Phi(\vec{u})}{|\vec{u}|^{nm}} \prod_{i=1}^m f_i\left(\frac{x}{|u_i|}\right) d\vec{u}.
$$

The fourth multilinear Hausdorff operator is modified to
$$
\widetilde{\dot{S}_{\Phi}}(\vec{f})(x) = \int_{\mathbb{R}^{nm}} \frac{\Phi\left(\frac{x}{|\vec{u}|}\right)}{|\vec{u}|^{nm}} \prod_{i=1}^m f_i(u_i) d\vec{u}.
$$

We formulate the main conclusions as follows.
\begin{theorem}
Assume that~$\Phi$~is a nonnegative, locally integrable, radial function.\\
Let~$1 \leq p, \tilde{p}, q \leq \infty$, $\alpha\in\mathbb{R}$ and (1.3) be satisfied. Then~$\widetilde{\mathcal{H}_{\Phi}}$~is bounded on the complementary local Morrey-type space~${}^cLML_{rad}^{\tilde{p},\lambda,q}L_{ang}^{p}(\mathbb{R}^n, |x|^{\alpha})$, that is,
$$
\|\widetilde{\mathcal{H}_{\Phi}}(f)\|_{{}^cLML_{rad}^{\tilde{p},\lambda,q}L_{ang}^{p}(\mathbb{R}^n, |x|^{\alpha})} \leq C_{\Phi,7}\|f\|_{{}^cLML_{rad}^{\tilde{p},\lambda,q}L_{ang}^{p}(\mathbb{R}^n, |x|^{\alpha})}
$$
if and only if
$$
C_{\Phi,7} = \omega_n \int_{0}^{\infty} \frac{\Phi(t)}{t^{\lambda-\frac{n+\alpha}{\tilde{p}} + 1}} dt < \infty.
$$
Moreover,
$$
\|\widetilde{\mathcal{H}_{\Phi}}\|_{{}^cLML_{rad}^{\tilde{p},\lambda,q}L_{ang}^{p}(\mathbb{R}^n, |x|^{\alpha}) \to {}^cLML_{rad}^{\tilde{p},\lambda,q}L_{ang}^{p}(\mathbb{R}^n, |x|^{\alpha})} = C_{\Phi,7}.
$$
\end{theorem}

\begin{theorem}
Assume that~$\Phi$~is a nonnegative locally integrable function.\\
Let~$1 \leq p, \tilde{p},q \leq \infty$, $\alpha\in\mathbb{R}$ and (1.3) be satisfied. Then~$\mathcal{H}_{\Phi}$~is bounded on the complementary local Morrey-type space~${}^cLML_{rad}^{\tilde{p},\lambda,q}L_{ang}^{p}(\mathbb{R}^n, |x|^{\alpha})$, that is,
$$
\|\mathcal{H}_{\Phi}(f)\|_{{}^cLML_{rad}^{\tilde{p},\lambda,q}L_{ang}^{p}(\mathbb{R}^n, |x|^{\alpha})} \leq C_{\Phi,8}\|f\|_{{}^cLML_{rad}^{\tilde{p},\lambda,q}L_{ang}^{p}(\mathbb{R}^n, |x|^{\alpha})}
$$
if and only if
$$
C_{\Phi,8} = \int_{\mathbb{R}^n} \frac{\Phi(y)}{|y|^{n - \frac{n+\alpha}{\tilde{p}} + \lambda}} dy < \infty.
$$
Moreover,
$$
\|\mathcal{H}_{\Phi}\|_{{}^cLML_{rad}^{\tilde{p},\lambda,q}L_{ang}^{p}(\mathbb{R}^n, |x|^{\alpha}) \to {}^cLML_{rad}^{\tilde{p},\lambda,q}L_{ang}^{p}(\mathbb{R}^n, |x|^{\alpha})} = C_{\Phi,8}.
$$
\end{theorem}

\begin{theorem}
Assume that~$\Phi$~is a nonnegative locally integrable function.\\
(i) Let~$1 \leq p, p_i, \tilde{p}, \tilde{p_i} < \infty$, $1 \leq q \leq \infty$ satisfy $\frac{1}{p} = \sum\limits_{i=1}^{m} \frac{1}{p_i}$ and $\frac{1}{\tilde{p}} = \sum\limits_{i=1}^{m} \frac{1}{\tilde{p_i}}$. Assume that $\alpha\in\mathbb{R}$, $\lambda = \sum\limits_{i=1}^{m} \lambda_i$~for~$i = 1, \ldots, m$, and $\lambda$, $\lambda_i$~satisfy
\[
\lambda, \lambda_i<0 \; if \; q<\infty \quad and\quad \lambda, \lambda_i\leq0 \; if \; q<\infty.
\tag{4.1}
\]
If
$$
C_{\Phi,9} = \int_{\mathbb{R}^{n_m}} \cdots \int_{\mathbb{R}^{n_2}} \int_{\mathbb{R}^{n_1}} \Phi(\vec{u}) \prod\limits_{i=1}^{m} |u_i|^{-n+\frac{n+\alpha}{\tilde{p_i}} - \lambda_i} d\vec{u} < \infty,
$$
then
$$
\|\dot{R}_{\Phi}(\vec{f})\|_{{}^cLML_{rad}^{\tilde{p},\lambda,q}L_{ang}^{p}(\mathbb{R}^n, |x|^{\alpha})} \leq C_{\Phi,9} \prod_{i=1}^{m} \|f_i\|_{{}^cLML_{rad}^{\tilde{p_i},\lambda,{(q\tilde{p_i})}/\tilde{p}}L_{ang}^{p_i}(\mathbb{R}^n, |x|^{\alpha})}.
$$
Moreover, if
$$
\lambda p = \lambda_i p_i, \quad i = 1, \ldots, m,
$$
then
$$
\|\dot{R}_{\Phi}\|_{\prod\limits_{i=1}^{m} {}^cLML_{rad}^{\tilde{p_i},\lambda,{(q\tilde{p_i})}/\tilde{p}}L_{ang}^{p_i}(\mathbb{R}^n, |x|^{\alpha}) \to {}^cLML_{rad}^{\tilde{p},\lambda,q}L_{ang}^{p}(\mathbb{R}^n, |x|^{\alpha})} = C_{\Phi,9}.
$$
(ii) Let~$1 \leq p, p_i, q, q_i < \infty$ satisfy $\frac{1}{p} = \sum\limits_{i=1}^{m} \frac{1}{p_i}$~and~$\frac{1}{q} = \sum\limits_{i=1}^m \frac{1}{q_i}$. Assume that $\alpha\in\mathbb{R}$, $\lambda, \lambda_i < 0$, and $\lambda = \sum\limits_{i=1}^{m} \lambda_i$~for~$i = 1, \ldots, m$. If
$$
\widetilde{C_{\Phi,9}} = \int_{\mathbb{R}^{n_m}} \cdots \int_{\mathbb{R}^{n_2}} \int_{\mathbb{R}^{n_1}} \frac{\Phi(\vec{u})}{\prod\limits_{i=1}^{m} |u_i|^{n + \lambda_i -\frac{\alpha}{p_i}}} d\vec{u} < \infty,
$$
then
$$
\|\dot{R}_{\Phi}(\vec{f})\|_{{}^cLML_{rad}^{\infty,\lambda,q}L_{ang}^{p}(\mathbb{R}^n, |x|^{\alpha})} \leq \widetilde{C_{\Phi,9}} \prod\limits_{i=1}^{m} \|f_i\|_{{}^cLML_{rad}^{\infty,\lambda_i,q_i}L_{ang}^{p_i}(\mathbb{R}^n, |x|^{\alpha})}.
$$
Moreover, if
$$
\lambda q = \lambda_i q_i, \quad i = 1, \ldots, m,
$$
then
$$
\|\dot{R}_{\Phi}\|_{\prod\limits_{i=1}^{m} {}^cLML_{rad}^{\infty,\lambda_i,q_i}L_{ang}^{p_i}(\mathbb{R}^n, |x|^{\alpha}) \to {}^cLML_{rad}^{\infty,\lambda,q}L_{ang}^{p}(\mathbb{R}^n, |x|^{\alpha})} = \widetilde{C_{\Phi,9}}.
$$
\end{theorem}

\begin{theorem}
Assume that~$\Phi$~is a nonnegative, locally integrable, radial function.
(i) Let~$1 \leq p, p_i, \tilde{p}, \tilde{p_i} < \infty$, $1 \leq q \leq \infty$ satisfy $\frac{1}{p} = \sum\limits_{i=1}^{m} \frac{1}{p_i}$ and $\frac{1}{\tilde{p}} = \sum\limits_{i=1}^{m} \frac{1}{\tilde{p_i}}$. Assume that $\alpha\in\mathbb{R}$, $\lambda = \sum\limits_{i=1}^{m} \lambda_i$~for~$i = 1, \ldots, m$, and $\lambda$, $\lambda_i$~satisfy (4.1). If
$$
C_{\Phi,10} = \omega_n^m \int_0^\infty \cdots \int_0^\infty \frac{\Phi(\vec{t})}{\prod\limits_{i=1}^m t_i^{\lambda_i - \frac{n+\alpha}{\tilde{p_i}} + 1}} d\vec{t} < \infty,
$$
then
$$
\|\widetilde{\dot{R}_\Phi}(\vec{f})\|_{{}^cLML_{rad}^{\tilde{p},\lambda,q}L_{ang}^{p}(\mathbb{R}^n, |x|^{\alpha})} \leq C_{\Phi,10} \prod_{i=1}^m \|f_i\|_{{}^cLML_{rad}^{\tilde{p_i},\lambda,{(q\tilde{p_i})}/\tilde{p}}L_{ang}^{p_i}(\mathbb{R}^n, |x|^{\alpha})}.
$$
Moreover, if
$$
\lambda p = \lambda_i p_i, \quad i = 1, \ldots, m,
$$
then
$$
\|\widetilde{\dot{R}_\Phi}\|_{\prod\limits_{i=1}^m {}^cLML_{rad}^{\tilde{p_i},\lambda,{(q\tilde{p_i})}/\tilde{p}}L_{ang}^{p_i}(\mathbb{R}^n, |x|^{\alpha}) \to {}^cLML_{rad}^{\tilde{p},\lambda,q}L_{ang}^{p}(\mathbb{R}^n, |x|^{\alpha})} = C_{\Phi,10}.
$$
(ii) Let~$1 \leq p, p_i, q, q_i < \infty$ satisfy $\frac{1}{p} = \sum\limits_{i=1}^{m} \frac{1}{p_i}$~and~$\frac{1}{q} = \sum\limits_{i=1}^m \frac{1}{q_i}$. Assume that $\alpha\in\mathbb{R}$, $\lambda, \lambda_i < 0$, and $\lambda = \sum\limits_{i=1}^{m} \lambda_i$~for~$i = 1, \ldots, m$. If
$$
\widetilde{C_{\Phi,10}} = \omega_n^m \int_0^\infty \cdots \int_0^\infty \frac{\Phi(\vec{t})}{\prod\limits_{i=1}^m t_i^{\lambda_i -\frac{\alpha}{p_i} + 1}} d\vec{t} < \infty,
$$
then
$$
\|\widetilde{\dot{R}_\Phi}(\vec{f})\|_{{}^cLML_{rad}^{\infty,\lambda,q}L_{ang}^{p}(\mathbb{R}^n, |x|^{\alpha})} \leq \widetilde{C_{\Phi,10}} \prod_{i=1}^m \|f_i\|_{{}^cLML_{rad}^{\infty,\lambda_i,q_i}L_{ang}^{p_i}(\mathbb{R}^n, |x|^{\alpha})}.
$$
Moreover, if
$$
\lambda q = \lambda_i q_i, \quad i = 1, \ldots, m,
$$
then
$$
\|\widetilde{\dot{R}_\Phi}\|_{\prod\limits_{i=1}^m {}^cLML_{rad}^{\infty,\lambda_i,q_i}L_{ang}^{p_i}(\mathbb{R}^n, |x|^{\alpha}) \to {}^cLML_{rad}^{\infty,\lambda,q}L_{ang}^{p}(\mathbb{R}^n, |x|^{\alpha})} = \widetilde{C_{\Phi,10}}.
$$
\end{theorem}

\begin{theorem}
Assume that~$\Phi$~is a nonnegative locally integrable function.\\
(i) Let~$1 \leq p,\tilde{p}, p_i, \tilde{p_i} < \infty$, $1 \leq q \leq \infty$ satisfy $\frac{1}{p} = \sum\limits_{i=1}^{m} \frac{1}{p_i}$ and $\frac{1}{\tilde{p}} = \sum\limits_{i=1}^{m} \frac{1}{\tilde{p_i}}$. Assume that $\alpha\in\mathbb{R}$, $\lambda = \sum\limits_{i=1}^{m} \lambda_i$~for~$i = 1, \ldots, m$, and $\lambda$, $\lambda_i$~satisfy (4.1). If
$$
C_{\Phi,11} = \int_{\mathbb{R}^{nm}} \frac{\Phi(\vec{u})}{|\vec{u}|^{nm}} \prod\limits_{i=1}^m |u_i|^{\frac{n+\alpha}{\tilde{p_i}} - \lambda_i} d\vec{u} < \infty,
$$
then
$$
\|\dot{S}_\Phi(\vec{f})\|_{{}^cLML_{rad}^{\tilde{p},\lambda,q}L_{ang}^{p}(\mathbb{R}^n, |x|^{\alpha})} \leq C_{\Phi,11} \prod_{i=1}^m \|f_i\|_{{}^cLML_{rad}^{\tilde{p_i},\lambda,{(q\tilde{p_i})}/\tilde{p}}L_{ang}^{p_i}(\mathbb{R}^n, |x|^{\alpha})}.
$$
Moreover, if
$$
\lambda p = \lambda_i p_i, \quad i = 1, \ldots, m,
$$
then
$$
\|\dot{S}_\Phi\|_{\prod\limits_{i=1}^m {}^cLML_{rad}^{\tilde{p_i},\lambda,{(q\tilde{p_i})}/\tilde{p}}L_{ang}^{p_i}(\mathbb{R}^n, |x|^{\alpha}) \to {}^cLML_{rad}^{\tilde{p},\lambda,q}L_{ang}^{p}(\mathbb{R}^n, |x|^{\alpha})} = C_{\Phi,11}.
$$
(ii) Let $1 \leq p, p_i, q, q_i < \infty$ satisfy $\frac{1}{p} = \sum\limits_{i=1}^{m} \frac{1}{p_i}$~and~$\frac{1}{q} = \sum\limits_{i=1}^m \frac{1}{q_i}$. Assume that $\alpha\in\mathbb{R}$, $\lambda, \lambda_i < 0$, and $\lambda = \sum\limits_{i=1}^{m} \lambda_i$~for~$i = 1, \ldots, m$. If
$$
\widetilde{C_{\Phi,11}} = \int_{\mathbb{R}^{nm}} \frac{\Phi(\vec{u})}{|\vec{u}|^{nm}} \prod_{i=1}^m |u_i|^{\frac{\alpha}{p_i}-\lambda_i} d\vec{u} < \infty,
$$
then
$$
\|\dot{S}_\Phi(\vec{f})\|_{{}^cLML_{rad}^{\infty,\lambda,q}L_{ang}^{p}(\mathbb{R}^n, |x|^{\alpha})} \leq \widetilde{C_{\Phi,11}} \prod_{i=1}^m \|f_i\|_{{}^cLML_{rad}^{\infty,\lambda_i,q_i}L_{ang}^{p_i}(\mathbb{R}^n, |x|^{\alpha})}.
$$
Moreover, if
$$
\lambda q = \lambda_i q_i, \quad i = 1, \ldots, m,
$$
then
$$
\|\dot{S}_\Phi\|_{\prod\limits_{i=1}^m {}^cLML_{rad}^{\infty,\lambda_i,q_i}L_{ang}^{p_i}(\mathbb{R}^n, |x|^{\alpha}) \to {}^cLML_{rad}^{\infty,\lambda,q}L_{ang}^{p}(\mathbb{R}^n, |x|^{\alpha})} = \widetilde{C_{\Phi,11}}.
$$
\end{theorem}

\begin{theorem}
Assume that~$\Phi$ is a nonnegative, locally integrable, radial function.
(i) Let~$1 \leq p,\tilde{p}, p_i, \tilde{p_i} < \infty$, $1 \leq q \leq \infty$ satisfy $\frac{1}{p} = \sum\limits_{i=1}^{m} \frac{1}{p_i}$ and $\frac{1}{\tilde{p}} = \sum\limits_{i=1}^{m} \frac{1}{\tilde{p_i}}$. Assume that $\alpha\in\mathbb{R}$, $\lambda = \sum\limits_{i=1}^{m} \lambda_i$~for~$i = 1, \ldots, m$, and $\lambda$, $\lambda_i$~satisfy (4.1). If
$$
C_{\Phi,12} = \omega_n^m \int_0^\infty \cdots \int_0^\infty \Phi \left( \frac{\prod\limits_{i=1}^m t_i}{\sqrt{\sum\limits_{i=1}^m t_i^2}} \right) \prod_{i=1}^m \frac{t_i^{ n(m-1) + \frac{n+\alpha}{\tilde{p_i}}-\lambda_i - 1}}{(\sum\limits_{i=1}^m t_i^2)^{\frac{nm}{2}}} d\vec{t} < \infty,
$$
then
$$
\|\widetilde{\dot{S}_\Phi}(\vec{f})\|_{{}^cLML_{rad}^{\tilde{p},\lambda,q}L_{ang}^{p}(\mathbb{R}^n, |x|^{\alpha})} \leq C_{\Phi,12} \prod_{i=1}^m \|f_i\|_{{}^cLML_{rad}^{\tilde{p_i},\lambda,{(q\tilde{p_i})}/\tilde{p}}L_{ang}^{p_i}(\mathbb{R}^n, |x|^{\alpha})}.
$$
Moreover, if
$$
\lambda p = \lambda_i p_i, \quad i = 1, \ldots, m,
$$
then
$$
\|\widetilde{\dot{S}_\Phi}\|_{\prod\limits_{i=1}^m {}^cLML_{rad}^{\tilde{p_i},\lambda,{(q\tilde{p_i})}/\tilde{p}}L_{ang}^{p_i}(\mathbb{R}^n, |x|^{\alpha}) \to {}^cLML_{rad}^{\tilde{p},\lambda,q}L_{ang}^{p}(\mathbb{R}^n, |x|^{\alpha})} = C_{\Phi,12}.
$$
(ii) Let~$1 \leq p, p_i, q, q_i < \infty$ satisfy $\frac{1}{p} = \sum\limits_{i=1}^{m} \frac{1}{p_i}$~and~$\frac{1}{q} = \sum\limits_{i=1}^m \frac{1}{q_i}$. Assume that $\alpha\in\mathbb{R}$, $\lambda, \lambda_i < 0$, and $\lambda = \sum\limits_{i=1}^{m} \lambda_i$~for~$i = 1, \ldots, m$. If
$$
\widetilde{C_{\Phi,12}} = \omega_{n}^m \int_0^\infty \cdots \int_0^\infty \Phi \left( \frac{\prod\limits_{i=1}^m t_i}{\sqrt{\sum\limits_{i=1}^m t_i^2}} \right) \prod_{i=1}^m \frac{t_i^{n(m-1) + \frac{\alpha}{p_i}-\lambda_i - 1}}{(\sum\limits_{i=1}^m t_i^2)^{\frac{nm}{2}}} d\vec{t} < \infty,
$$
then
$$
\|\widetilde{\dot{S}_\Phi}(\vec{f})\|_{{}^cLML_{rad}^{\infty,\lambda,q}L_{ang}^{p}(\mathbb{R}^n, |x|^{\alpha})} \leq \widetilde{C_{\Phi,12}} \prod_{i=1}^m \|f_i\|_{{}^cLML_{rad}^{\infty,\lambda_i,q_i}L_{ang}^{p_i}(\mathbb{R}^n, |x|^{\alpha})}.
$$
Moreover, if
$$
\lambda q = \lambda_i q_i, \quad i = 1, \ldots, m,
$$
then
$$
\|\widetilde{\dot{S}_\Phi}\|_{\prod\limits_{i=1}^m {}^cLML_{rad}^{\infty,\lambda_i,q_i}L_{ang}^{p_i}(\mathbb{R}^n, |x|^{\alpha}) \to {}^cLML_{rad}^{\infty,\lambda,q}L_{ang}^{p}(\mathbb{R}^n, |x|^{\alpha})} = \widetilde{C_{\Phi,12}}.
$$
\end{theorem}

The proof methods for the above theorems are similar to the ideas of the conclusions in section 2. For the sake of simplicity, we will omit the details. In addition,
We also can give the related applications by imitating that of examples in sections 3. Due to the similarity of the conclusions, we will not repeat them in this position.

{\bf Acknowledgements} The authors cordially thank the anonymous referees who gave valuable
suggestions and useful comments which have led to the improvement of this paper. This work is supported by the National Natural Science Foundation of China (No. 12501128).

{\bf Declaration}

{\bf Conflict of interest} The authors declare that they have no conflict of interest.


\begin{thebibliography}{99}


\bibitem{alvarez2000} J. Alvarez, J. Lakey and M. Guzm\'{a}n-Partida,
Spaces of bounded \(\lambda\)-central mean oscillation, Morrey spaces, and \(\lambda\)-central Carleson measures,
Collect. Math. {\bf 51} (2000), no. 1, 1-47.

\bibitem{andersen2003} K. Andersen,
Boundedness of Hausdorff operators on \(L^p(\mathbb{R}^n)\), \(H^1(\mathbb{R}^n)\), and \(\text{BMO}(\mathbb{R}^n)\),
Acta Sci. Math. (Szeged) {\bf 69} (2003), no. 1-2, 409-418.

\bibitem{An2023}N. An, M. Wei and F. Zhao,
Sharp constants for multilinear Hausdorff operators on local Morrey-type spaces,
Forum Math. {\bf 35} (2023), no.~4, 1133--1154.

\bibitem{brown2002} G. Brown and F. M¨®ricz,
Multivariate Hausdorff operators on the spaces \(L^p(\mathbb{R}^n)\),
J. Math. Anal. Appl. {\bf 271} (2002), no. 2, 443-454.

\bibitem{burenkov2004} V. Burenkov and H. Guliyev,
Necessary and sufficient conditions for boundedness of the maximal operator in local Morrey-type spaces,
Studia Math. {\bf 163} (2004), no. 2, 157-176.

\bibitem{VHV2007}V. Burenkov, H. Guliyev, V. Guliyev,
On boundedness of the fractional maximal operator from complementary Morrey-type spaces to Morrey-type spaces,
in: The Interaction of Analysis and Geometry, in: Contemp. Math., vol. 424, Am. Math. Soc., Providence, (2007), pp. 17--32.

\bibitem{carton1984} C. Carton-Lebrun and M. Fosset,
Moyennes et quotients de Taylor dans BMO,
Bull. Soc. Roy. Sci. Li¨¨ge {\bf 53} (1984), no. 2, 85-87.

\bibitem{chen2012} J. Chen, D. Fan and J. Li,
Hausdorff operators on function spaces,
Chinese Ann. Math. Ser. B {\bf 33} (2012), no. 4, 537-556.

\bibitem{chen2012multilinear} J. Chen, D. Fan and C. Zhang,
Multilinear Hausdorff operators and their best constants,
Acta Math. Sin. (Engl. Ser.) {\bf 28} (2012), no. 8, 1521-1530.

\bibitem{chen2013} J. Chen, D. Fan and S. Wang,
Hausdorff operators on Euclidean spaces,
Appl. Math. J. Chinese Univ. Ser. B {\bf 28} (2013), no. 4, 548-564.

\bibitem{chen2014} J. Chen and X. Zhu,
Boundedness of multidimensional Hausdorff operators on \(H^1(\mathbb{R}^n)\),
J. Math. Anal. Appl. {\bf 409} (2014), no. 1, 428-434.

\bibitem{chen2016} J. Chen, D. Fan, X. Lin and J. Ruan,
The fractional Hausdorff operators on the Hardy spaces \(H^p(\mathbb{R}^n)\),
Anal. Math. {\bf 42} (2016), no. 1, 1-17.

\bibitem{CL}F. Cacciafesta and R. Luc\`{a},
Singular integrals with angular regularity,
Proc. Amer. Math. Soc. {\bf 144} (2016), no. 8, 3413--3418.

\bibitem{DL1}P. D'Ancona and R. Luc\`{a},
Stein-Weiss and Caffarelli-Kohn-Nirenberg inequalities with higher angular integrability,
J. Math. Anal. Appl. {\bf 388} (2012), no. 2, 1061--1079.

\bibitem{DL2}P. D'Ancona and R. Luc\`{a},
On the regularity set and angular integrability for the Navier-Stokes equation,
Arch. Rational Mech. Anal. {\bf 221} (2016), no. 3, 1255--1284.

\bibitem{fan2014} D. Fan and F. Zhao,
Multilinear fractional Hausdorff operators,
Acta Math. Sin. (Engl. Ser.) {\bf 30} (2014), no. 8, 1407-1421.

\bibitem{fan2019} D. Fan and F. Zhao,
Sharp constants for multivariate Hausdorff \(q\)-inequalities,
J. Aust. Math. Soc. {\bf 106} (2019), no. 2, 274-286.

\bibitem{fu2012} Z. Fu, L. Grafakos, S. Lu and F. Zhao,
Sharp bounds for \(m\)-linear Hardy and Hilbert operators,
Houston J. Math. {\bf 38} (2012), no. 1, 225-244.


\bibitem{fu2015} Z. Fu, S. Gong, S. Lu and W. Yuan,
 Weighted multilinear Hardy operators and commutators,
Forum Math. {\bf 27} (2015), no. 5, 2825-2851.


\bibitem{gao2015} G. Gao and Y. Fan,
Sharp bounds for multilinear Hausdorff operators,
Acta Math. Sinica (Chinese Ser.) {\bf 58} (2015), no. 1, 153-160.



\bibitem{guliyev2017} V. Guliyev, S. Hasanov and Y. Sawano,
Decompositions of local Morrey-type spaces,
Positivity {\bf 21} (2017), no. 3, 1223-1252.

\bibitem{hausdorff1921} F. Hausdorff,
Summationsmethoden und Momentfolgen, I,
Math. Z. {\bf 9} (1921), no. 1-2, 74-109.

\bibitem{hurwitz1917} W. Hurwitz and L. Silverman,
On the consistency and equivalence of certain definitions of summability,
Trans. Amer. Math. Soc. {\bf 18} (1917), no. 1, 1-20.

\bibitem{karapetyants2020} A. Karapetyants and E. Liflyand,
Defining Hausdorff operators on Euclidean spaces,
Math. Methods Appl. Sci. {\bf 43} (2020), no. 16, 9487-9498.


\bibitem{lerner2007} A. Lerner and E. Liflyand,
Multidimensional Hausdorff operators on the real Hardy space,
J. Aust. Math. Soc. {\bf 83} (2007), no. 1, 79-86.

\bibitem{liflyand2013} E. Liflyand,
Hausdorff operators on Hardy spaces,
Eurasian Math. J. {\bf 4} (2013), no. 4, 101-141.

\bibitem{LF}F. Liu and D. Fan,
Weighted estimates for rough singualr integrals with applications to angular integrability,
Pacific J. Math. {\bf 301} (2019), no. 1, 267--295.

\bibitem{LLW}
R. Liu, F Liu and H. Wu,
Mixed radial-angular integrability for rough singular integrals and maximal operators,
Proc. Amer. Math. Soc. {\bf 148} (2020), no. 9, 3943--3956.

\bibitem{LW}
R. Liu and H. Wu,
Rough singular integrals and maximal operator with radial-angular integrability,
Proc. Amer. Math. Soc. {\bf 150} (2022), no. 3, 1141--1151.

\bibitem{RSH2023}R. Liu, S. Tao and H. Wu,
Characterizations of the mixed radial-angular central Campanato space via the commutators of Hardy type,
Forum Math. {\bf 35} (2023), no. 5, 1327--1346

\bibitem{RYS2024}R. Liu, Y. Yang and S. Tao,
Some sharp bounds for Hardy-type operators on mixed radial-angular type function spaces,
Ann. Funct. Anal. {\bf 15} (2024), no. 4, Paper No. 84, 38 pp.

\bibitem{RSH2025}R. Liu, S. Tao and H. Wu,
Mixed radial-angular integrabilities for commutators of fractional Hardy operators with rough kernels,
Anal. Math. Phys. {\bf 15} (2025), no. 2, Paper No. 40, 16 pp.

\bibitem{LZ2025}
R. Liu and Q. Zhang,
Sharp bounds for the $m$-linear Hardy operator on mixed radial-angular central total Morrey spaces, Submitted.

\bibitem{moricz2005} F. M\'{o}ricz,
Multivariate Hausdorff operators on the spaces \(H^1(\mathbb{R}^n)\) and \(\text{BMO}(\mathbb{R}^n)\),
Anal. Math. {\bf 31} (2005), no. 1, 31-41.

\bibitem{ruan2016} J. Ruan and D. Fan,
Hausdorff operators on the power weighted Hardy spaces,
J. Math. Anal. Appl. {\bf 433} (2016), no. 1, 31-48.

\bibitem{wu2015} X. Wu,
Necessary and sufficient conditions for generalized Hausdorff operators and commutators,
Ann. Funct. Anal. {\bf 6} (2015), no. 3, 60-72.

\bibitem{Wei2025}M. Wei, X. Liu and D. Yan,
Sharp estimates for Hausdorff operators on complementary local Morrey-type spaces,
Bull. Sci. Math. {\bf 201} (2025), Paper No. 103614, 44 pp.

\bibitem{xiao2001} J. Xiao,
\(L^p\) and \(\text{BMO}\) bounds of weighted Hardy-Littlewood averages,
J. Math. Anal. Appl. {\bf 262} (2001), no. 2, 660-666.

\end{thebibliography}
\end{document}